\documentclass[letterpaper, 10 pt, conference]{ieeeconf} % Comment this line out if you need a4paper
\IEEEoverridecommandlockouts % This command is only needed if 
\usepackage{graphics} % for pdf, bitmapped graphics files
\usepackage{epsfig} % for postscript graphics files
\usepackage{amsmath} % assumes amsmath package installed
\usepackage{amssymb} % assumes amsmath package installed
\usepackage{subfigure,color,hyperref}
\usepackage{gensymb}
\newcommand{\xb}{\mathbf{x}}
\newcommand{\diag}{\mathop{\mathrm{diag}}}

\newcommand{\trs}{\mathop{\mathrm{tr}}}
\newcommand{\norm}[1]{\ensuremath{\left\| #1 \right\|}}

\newcommand{\refeqn}[1]{(\ref{eqn:#1})}

\newcommand{\SO}{\ensuremath{\mathsf{SO(3)}}}

\newcommand{\so}{\ensuremath{\mathfrak{so}(3)}}

\renewcommand{\Re}{\ensuremath{\mathbb{R}}}
\newcommand{\Sph}{\ensuremath{\mathsf{S}}}

\newcommand{\g}{\ensuremath{\mathfrak{g}}}
\usepackage{dsfont}

\newtheorem{prop}{Proposition}

\newcommand{\Pb}{\mathbf{P}}
\newcommand{\Mb}{\mathbf{M}}
\newcommand{\Nb}{\mathbf{N}}

\newcommand{\Bb}{\mathbf{B}}

\newcommand{\Gb}{\mathbf{G}}
\newcommand{\zb}{\mathbf{z}}

\title{\LARGE \bf Dynamics and Control of Quadrotor UAVs Transporting a Rigid Body Connected via Flexible Cables}

\author{Farhad A. Goodarzi and Taeyoung Lee$^{*}$% <-this % stops a space
\thanks{Farhad A. Goodarzi and Taeyoung Lee, Mechanical and Aerospace Engineering, The George
Washington University, Washington DC 20052
{\tt\small {\{fgoodarzi,tylee}\}@gwu.edu}}%
\thanks{$^*$This research has been supported in part by NSF under the grants CMMI-1243000 (transferred from 1029551), CMMI-1335008, and CNS-1337722.}}

\begin{document}
\maketitle
\pagestyle{empty}

\begin{abstract}
This paper is focused on the dynamics and control of arbitrary number of quadrotor UAVs transporting a rigid body payload. The rigid body payload is connected to quadrotors via flexible cables where each flexible cable is modeled as a system of serially-connected links. It is shown that a coordinate-free form of equations of motion can be derived for arbitrary numbers of quadrotors and links according to Lagrangian mechanics on a manifold. A geometric nonlinear controller is presented to transport the rigid body to a fixed desired position while aligning all of the links along the vertical direction. Numerical results are provided to illustrate the desirable features of the proposed control system.
\end{abstract}

\section{introduction}

There are various applications for aerial load transportation such as usage in construction, military operations, emergency response, or delivering packages. Load transportation with the cable-suspended load has been studied traditionally for a helicopter~\cite{CicKanJAHS95,BerPICRA09} or for small unmanned aerial vehicles such as quadrotor UAVs~\cite{PalCruIRAM12,MicFinAR11,MazKonJIRS10}. 

In most of the prior works, the dynamics of aerial transportation has been simplified due to the inherent dynamic complexities. For example, it is assumed that the dynamics of the payload is considered completely decoupled from quadrotors, and the effects of the payload and the cable are regarded as arbitrary external forces and moments exerted to the quadrotors~\cite{ ZamStaJDSMC08, SchMurIICRA12, PalFieIICRA12}, thereby making it challenging to suppress the swinging motion of the payload actively, particularly for agile aerial transportations.

%the payload considered as a point mass and the dynamics of the payload are ignored in equation of motion, or the cables connecting the payload to the UAV are considered to be always taut and rigid~\cite{LeeSrePICDC13}. In other words, the effect of the payload and cable is considered as external forces and torques to the UAV, instead of considering the dynamic coupling between the UAV and the payload

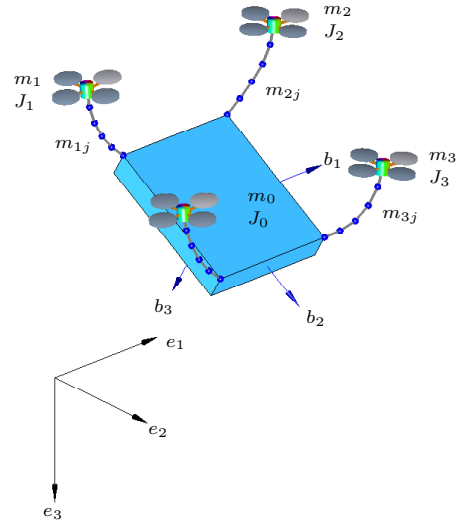
\begin{figure}
\centerline{
	\setlength{\unitlength}{0.09\columnwidth}\scriptsize
\begin{picture}(5,8.8)(0,0)
\put(0,0){\includegraphics[width=.6\columnwidth]{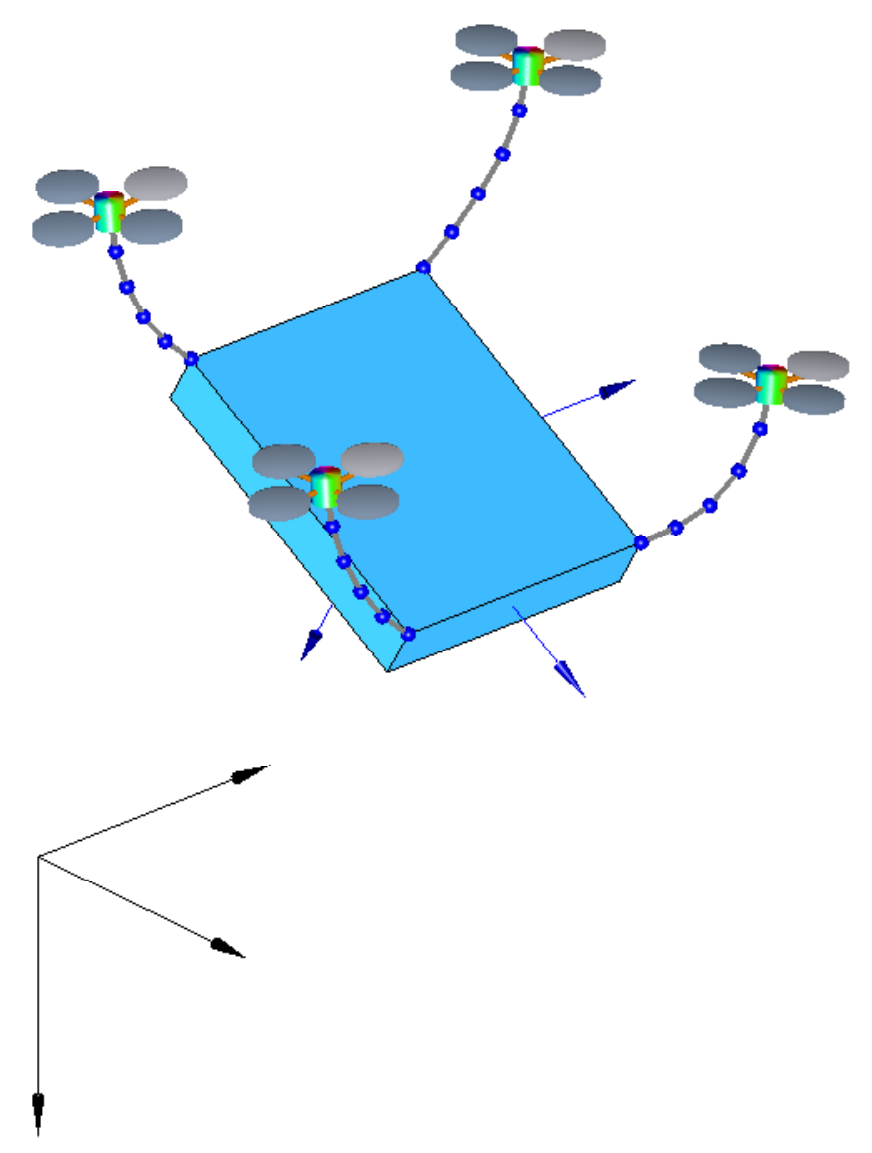}}
\put(3.6,5.3){\shortstack[c]{$m_{0}$}}
\put(3.6,4.9){\shortstack[c]{$J_{0}$}}
\put(-0.4,7.3){\shortstack[c]{$m_{1}$}}
\put(0.3,6.2){\shortstack[c]{$m_{1j}$}}
\put(-0.4,6.9){\shortstack[c]{$J_{1}$}}
\put(4.9,8.5){\shortstack[c]{$m_{2}$}}
\put(3.9,7.1){\shortstack[c]{$m_{2j}$}}
\put(4.9,8.1){\shortstack[c]{$J_{2}$}}
\put(6.7,6.0){\shortstack[c]{$m_{3}$}}
\put(5.9,5.0){\shortstack[c]{$m_{3j}$}}
\put(6.7,5.6){\shortstack[c]{$J_{3}$}}
\put(2.2,2.8){\shortstack[c]{$e_{1}$}}
\put(1.9,1.3){\shortstack[c]{$e_{2}$}}
\put(0.1,-0.1){\shortstack[c]{$e_{3}$}}
\put(4.9,5.9){\shortstack[c]{$b_{1}$}}
\put(4.6,3.2){\shortstack[c]{$b_{2}$}}
\put(2.0,3.4){\shortstack[c]{$b_{3}$}}
\end{picture}
}
\caption{Quadrotor UAVs with a rigid body payload. Cables are modeled as a serial connection of an arbitrary number of links (only 4 quadrotors with 5 links in each cable are illustrated).}\label{fig:fig1}
\end{figure}

Recently, the coupled dynamics of the payload or cable has been explicitly incorporated into control system design~\cite{LeeSrePICDC13}. In particular, a complete model of a quadrotor transporting a payload modeled as a point mass, connected via a flexible cable is presented, where the cable is modeled as serially connected links to represent the deformation of the cable~\cite{gooddaewontaeyoungacc14}. In another distinct study, multiple quadrotors transporting  a rigid body payload has been studied~\cite{LeeMultipleRigid14}, but it is assume that the cables connecting the rigid body payload and quadrotors are always taut. These assumptions and simplifications in the dynamics of the system reduce the stability of the controlled system, particularly in rapid and aggressive load transportation where the motion of the cable and payload is excited nontrivially.

The first distinct contribution of this paper is presenting the complete dynamic model of an arbitrary number of quadrotors transporting a rigid body where each quadrotor is connected to the rigid body via a flexible cable. Each flexible cable is modeled as an arbitrary number of serially connected links, and it is valid for various masses and lengths. A coordinate free form of equations of motion is derived according to Lagrange mechanics on a nonlinear manifold for the full dynamic model. These sets of equations of motion are presented in a complete and organized manner without any simplification.

%Based on the best knowledge of the authors, such complete and accurate model of quadrotor transporting a rigid body with considering deformation of cables has been never presented.

Another contribution of this study is designing a control system to stabilize the rigid body at desired position. Geometric nonlinear controllers presented in the author's previous study is utilized~\cite{LeeLeoPICDC10,LeeLeoAJC13,Farhad2013}, and they are generalized for the presented model. More explicitly, we show that the rigid body payload is asymptotically transported into a desired location, while aligning all of the links along the vertical direction corresponding to a hanging equilibrium.

%the hanging equilibrium of the links can be asymptotically stabilized while translating the payload to a desired position. %In contrast to existing papers where the force and the moment exerted by the payload to the quadrotor are considered as disturbances.

The unique property of the proposed control system is that the nontrivial coupling effects between the dynamics of rigid payload, flexible cables, and multiple quadrotors are explicitly incorporated into control system design, without any simplifying assumption. Another distinct feature is that the equations of motion and the control systems are developed directly on the nonlinear configuration manifold intrinsically. Therefore, singularities of local parameterization are completely avoided to generate agile maneuvers of the payload in a uniform way. In short, the proposed control system is particularly useful for rapid and safe payload transportation in complex terrain, where the position of the payload should be controlled concurrently while suppressing the deformation of the cables.

This paper is organized as follows. A dynamic model is presented and the problem is formulated at Section II. Control systems are constructed at Sections III and IV, which are followed by numerical examples in Section V. Due to the page limit, parts of proofs are relegated to~\cite{FarhadTLeeMultiple}.

%%%%%%%%%%%%%%%%%%%%%%%%%%%%%%%%%%%%%%%%%%%%%%%%%%%%%%%%%%%%%%%
%%%%%%%%%%%%%%%%%%%%%%%%%%%%%%%%%%%%%%%%%%%%%%%%%%%%%%%%%%%%%%%
\section{Problem Formulation}

Consider a rigid body with mass $m_{0}\in\Re$ and moment of inertia $J_{0}\in\Re^{3\times 3}$, being transported with arbitrary $n$ numbers of quadrotors. The location of the mass center of the rigid body is denoted by $x_{0}\in\Re^{3}$, and its attitude is given by $R_{0}\in\SO$, where the special orthogonal group is given by $\SO=\{R\in\Re^{3\times 3} \mid R^{T}R=I,\det(R)=1\}$. Figure \ref{fig:fig1} illustrates the system with an inertial frame. We choose an inertial frame $\{\vec{e}_{1},\vec{e}_{2},\vec{e}_{3}\}$ and body fixed frame $\{\vec{b}_{1},\vec{b}_{2},\vec{b}_{3}\}$ attached to the payload. We also consider a body fixed frame attached to the $i$-th quadrotor $\{\vec{b}_{1_i},\vec{b}_{2_i},\vec{b}_{3_i}\}$. In the inertial frame, the third axes $\vec{e}_{3}$ points downward with gravity and the other axes are chosen to form an orthonormal frame. 

The mass and the moment of inertia of the $i$-th quadrotor are denoted by $m_{i}\in\Re$ and $J_{i}\in\Re^{3\times 3}$ respectively. The cable connecting each quadrotor to the rigid body is modeled as an arbitrary numbers of links for each quadrotor with varying masses and lengths. The direction of the $j$-th link of the $i$-th quadrotor, measured outward from the quadrotor toward the payload is defined by the unit vector $q_{ij}\in\Sph^2$, where $\Sph^2=\{q\in\Re^{3}\mid\|q\|=1\}$, where the mass and length of that link is denoted with $m_{ij}$ and $l_{ij}$ respectively. The number of links in the cable connected to the $i$-th quadrotor is defined as $n_{i}$.

The configuration manifold for this system is given by $\SO\times \Re^{3}\times (\SO^{n})\times (\Sph^2)^{\sum_{i=1}^{n}n_{i}}$. The $i$-th quadrotor can generate a thrust force of $-f_{i}R_{i}e_{3}\in\Re^{3}$ with respect to the inertial frame, where $f_{i}\in\Re$ is the total thrust magnitude of the $i$-th quadrotor. It also generates a moment $M_{i}\in\Re^{3}$ with respect to its body-fixed frame.
Throughout this paper, the two norm of a matrix $A$ is denoted by $\|A\|$. The standard dot product is denoted by $x\cdot y=x^{T}y$ for any $x,y\in\Re^{3}$.
\subsection{Lagrangian}
The kinematics equations for the links, payload, and quadrotors are given by
\begin{gather}
\dot{q}_{ij}=\omega_{ij}\times q_{ij}=\hat{\omega}_{ij}q_{ij},\\
\dot{R}_{0}=R_{0}\hat{\Omega}_{0},\\
\dot{R}_{i}=R_{i}\hat{\Omega}_{i},
\end{gather}
where $\omega_{ij}\in\Re^{3}$ is the angular velocity of the $j$-th link in the $i$-th cable satisfying $q_{ij}\cdot\omega_{ij}=0$. Also, $\Omega_{0}\in\Re^{3}$ is the angular velocity of the payload and $\Omega_{i}\in\Re^{3}$ is the angular velocity of the $i$-th quadrotor, expressed with respect to the corresponding body fixed frame. The hat map $\hat{\cdot}:\Re^{3}\rightarrow\so$ is defined by the condition that $\hat{x}y=x\times y$ for all $x,y\in\Re^{3}$, and the inverse of the hat map is denoted by the vee map $\vee:\so\rightarrow\Re^{3}$.

The position of the $i$-th quadrotor is given by
\begin{align}\label{eqn:xi}
x_{i}=x_{0}+R_{0}\rho_{i}-\sum_{a=1}^{n_{i}}{l_{ia}q_{ia}},
\end{align}
where $\rho_{i}\in\Re^{3}$ is the vector from the center of mass of the rigid body to the point that $i$-th quadrotor is connected to rigid body via the cable. Similarly the position of the $j$-th link in the cable connecting the $i$-th quadrotor to the rigid body is given by
\begin{align}\label{eqn:xij}
x_{ij}=x_{0}+R_{0}\rho_{i}-\sum_{a=j+1}^{n_{i}}{l_{ia}q_{ia}}.
\end{align}

We derive equations of motion according to Lagrangian mechanics. Total kinetic energy of the system is given by
\begin{align}
T=&\frac{1}{2}m_{0}\|\dot{x}_{0}\|^{2}+\sum_{i=1}^{n}\sum_{j=1}^{n_{i}}{\frac{1}{2}m_{ij}\|\dot{x}_{ij}\|^{2}}+\frac{1}{2}\sum_{i=1}^{n}{m_{i}\|\dot{x}_{i}\|^{2}}\nonumber\\
&+\frac{1}{2}\sum_{i=1}^{n}{\Omega_{i}\cdot J_{i}\Omega_{i}}+\frac{1}{2}\Omega_{0}\cdot J_{0}\Omega_{0}.
\end{align}
The gravitational potential energy is given by
\begin{align}
V=-m_{0}ge_{3}\cdot x_{0}-\sum_{i=1}^{n}{m_{i}ge_{3}}\cdot x_{i}-\sum_{i=1}^{n}\sum_{j=1}^{n_{i}}{m_{ij}ge_{3}}\cdot x_{ij},
\end{align}
where it is assumed that the unit-vector $e_{3}$ points downward along the gravitational acceleration as shown at Figure \ref{fig:fig1}. The corresponding Lagrangian of the system is $L=T-V$.

\subsection{Euler-Lagrange equations}
Coordinate-free form of Lagrangian mechanics on the two-sphere $\Sph^2$ and the special orthogonal group $\SO$ for various multibody systems has been studied in~\cite{Lee08,LeeLeoIJNME08}. The key idea is representing the infinitesimal variation of $R_i\in\SO$ in terms of the exponential map
\begin{align}
\delta R_{i} = \frac{d}{d\epsilon}\bigg|_{\epsilon = 0} R_{i}\exp(\epsilon \hat\eta_{i}) = R_{i}\hat\eta_{i},\label{eqn:delR}
\end{align}
for $\eta_{i}\in\Re^3$. The corresponding variation of the angular velocity is given by $\delta\Omega_{i}=\dot\eta_{i}+\Omega_{i}\times\eta_{i}$. Similarly, the infinitesimal variation of $q_{ij}\in\Sph^2$ is given by
\begin{align}
\delta q_{ij} = \xi_{ij}\times q_{ij},\label{eqn:delqi}
\end{align}
for $\xi_{ij}\in\Re^3$ satisfying $\xi_{ij}\cdot q_{ij}=0$. This lies in the tangent space as it is perpendicular to $q_{i}$. Using these, we obtain the following Euler-Lagrange equations.

\begin{prop}\label{prop:FDM}
By using the above expressions, the equations of motion can be obtained from Hamilton's principle:
\begin{gather}\label{eqn:EOM}
M_{T}\ddot{x}_{0}-\sum_{i=1}^{n}\sum_{j=1}^{n_{i}}{M_{0ij}l_{ij}\ddot{q}_{ij}}-\sum_{i=1}^{n}{M_{iT}R_{0}\hat{\rho}_{i}\dot{\Omega}_{0}} \nonumber\\
=M_{T}ge_{3}+\sum_{i=1}^{n}{-f_{i}R_{i}e_{3}}-\sum_{i=1}^{n}{M_{iT}R_{0}\hat{\Omega}_{0}^2 \rho_{i}},\label{eqn:EOMM1}\\
\bar{J}_{0}\dot{\Omega}_{0}+\sum_{i=1}^{n}{M_{iT}\hat{\rho}_{i}R_{0}^{T}\ddot{x}_{0}}-\sum_{i=1}^{n}\sum_{j=1}^{n_{i}}{M_{0ij}l_{ij}\hat{\rho}_{i}R_{0}^{T}\ddot{q}_{ij}} \nonumber\\
=\sum_{i=1}^{n}{\hat{\rho}_{i}R_{0}^{T}(-f_{i}R_{i}e_{3}+M_{iT}ge_{3})}-\hat{\Omega}_{0}\bar{J}_{0}\Omega_{0},\label{eqn:EOMM2}\\
\sum_{k=1}^{n_{i}}{M_{0ij}l_{ik}\hat{q}_{ij}^{2}\ddot{q}_{ik}}-M_{0ij}\hat{q}_{ij}^{2}\ddot{x}_{0}+M_{0ij}\hat{q}_{ij}^{2}R_{0}\hat{\rho}_{i}\dot{\Omega}_{0} \nonumber\\
=M_{0ij}\hat{q}_{ij}^{2}R_{0}\hat{\Omega}_{0}^{2}\rho_{i}-\hat{q}_{ij}^{2}(M_{0ij}ge_{3}-f_{i}R_{i}e_{3}),\label{eqn:EOMM3}\\
J_{i}\Omega_{i}+\Omega_{i}\times J_{i}\Omega_{i}=M_{i}\label{eqn:EOMM4}.
\end{gather}
Here the total mass $M_{T}$ of the system and the mass of the $i$-th quadrotor and its flexible cable $M_{iT}$ are defined as
\begin{gather}
M_{T}=m_{0}+\sum_{i=1}^{n}M_{iT},\; M_{iT}=\sum_{j=1}^{n_{i}}{m_{ij}}+m_{i},\label{eqn:def1}
\end{gather}
and the constants related to the mass of links are given as
\begin{align}
M_{0ij}&=m_{i}+\sum_{a=1}^{j-1}{m_{ia}}\label{eqn:def3},
\end{align}
The equations of motion can be rearranged in a matrix form as follow
\begin{align}
\Nb\ddot{X}=\Pb
\end{align}
where the state vector $X\in\Re^{D_{X}}$ with $D_{X}=6+3\sum_{i=1}^{n}n_{i}$ is given by
\begin{align}
X=[{x}_{0},\; {\Omega}_{0},\; {q}_{1j},\; {q}_{2j},\; \cdots,\; {q}_{nj}]^{T},
\end{align}
and matrix $\Nb\in\Re^{D_{X}\times D_{X}}$ is defined as
\begin{align}\label{eqn:EOM11}
\Nb=\begin{bmatrix}
M_{T}I_{3}&\Nb_{x_{0}\Omega_{0}}&\Nb_{x_{0}1}&\Nb_{x_{0}2}&\cdots&\Nb_{x_{0}n}\\
\Nb_{\Omega_{0} x_{0}}&\bar{J}_{0}&\Nb_{\Omega_{0}1}&\Nb_{\Omega_{0}2}&\cdots&\Nb_{\Omega_{0}n}\\
\Nb_{1 x_{0}}&\Nb_{1\Omega_{0}}&\Nb_{qq1}&0&\cdots&0\\
\Nb_{2 x_{0}}&\Nb_{2\Omega_{0}}&0&\Nb_{qq2}&\cdots&0\\
\vdots&\vdots&\vdots&\vdots&\vdots&\vdots\\
\Nb_{n x_{0}}&\Nb_{n\Omega_{0}}&0&0&\cdots&\Nb_{qqn}
\end{bmatrix},
\end{align}
where the sub-matrices are defined as
\begin{gather}
\Nb_{x_{0}\Omega_{0}}=-\sum_{i=1}^{n}{M_{iT}R_{0}\hat{\rho}_{i}};\; \Nb_{\Omega_{0} x_{0}}=\Mb_{x_{0}\Omega_{0}}^{T},\nonumber\\
\Nb_{x_{0}i}=-[M_{0i1}l_{i1}{I}_{3},\; M_{0i2}l_{i2}{I}_{3},\; \cdots,\;M_{0in_{i}}l_{in_{i}}{I}_{3}],\nonumber\\
\Nb_{\Omega_{0}i}=-[M_{0i1}l_{i1}\hat{\rho}_{i}R_{0}^{T},\; M_{0i2}l_{i2}\hat{\rho}_{i}R_{0}^{T},\; \cdots,\; M_{0in_{i}}l_{in_{i}}\hat{\rho}_{i}R_{0}^{T}],\nonumber\\
\Nb_{ix_{0}}=-[M_{0i1}\hat{q}_{i1}^{2},\; M_{0i2}\hat{q}_{i2}^{2},\; \cdots,\;M_{0in_{i}}\hat{q}_{in_{i}}^{2}]^{T},\nonumber\\
\Nb_{i\Omega_{0}}=[M_{0i1}\hat{q}_{i1}^{2}R_{0}\hat{\rho}_{i},\; M_{0i2}\hat{q}_{i2}^{2}R_{0}\hat{\rho}_{i},\; \cdots,\; M_{0in_{i}}\hat{q}_{in_{i}}^{2}R_{0}\hat{\rho}_{i}]^{T},
\end{gather}
and the sub-matrix $\Nb_{qqi}\in\Re^{3n_i\times 3n_i}$ is given by
\begin{align}\Nb_{qqi}=
\begin{bmatrix}
-M_{011}l_{i1}I_{3}&M_{012}l_{i2}\hat{q}_{i2}^2&\cdots&M_{01n_{i}}l_{in_{i}}\hat{q}_{in_{i}}^2\\
M_{021}l_{i1}\hat{q}_{i1}^2&-M_{022}l_{i2}I_{3}&\cdots&M_{02n_{i}}l_{in_{i}}\hat{q}_{in_{i}}^2\\
\vdots&\vdots&&\vdots\\
M_{0n_{i}1}l_{i1}\hat{q}_{i1}^2&M_{0n_{i}2}l_{i2}\hat{q}_{i2}^2&\cdots&-M_{0n_{i}n_{i}}l_{in_{i}}I_{3}
\end{bmatrix}.
\end{align}
The $\Pb\in\Re^{D_{X}}$ matrix is
\begin{align}
\Pb=[P_{x_{0}},\; P_{\Omega_{0}},\; P_{1j},\; P_{2j},\; \cdots,\; P_{nj}]^{T},
\end{align}
and sub-matrices of $\Pb$ matrix are also defined as 
\begin{align*}
P_{x_{0}}&=M_{T}ge_{3}+\sum_{i=1}^{n}{-f_{i}R_{i}e_{3}}-\sum_{i=1}^{n}{M_{iT}R_{0}\hat{\Omega}_{0}^2\rho_{i}},\\
P_{\Omega_{0}}&=-\hat{\Omega}_{0}\bar{J}_{0}\Omega_{0}+\sum_{i=1}^{n}{\hat{\rho}_{i}R_{0}^T(M_{iT}ge_{3}-f_{i}R_{i}e_{3})},\nonumber\\
P_{ij}=&-\hat{q}_{ij}^2(-f_{i}R_{i}e_{3}+M_{0ij}ge_{3})+M_{0ij}\hat{q}_{ij}^2 R_{0}\hat{\Omega}_{0}^2 \rho_{i}\\
&+M_{0ij}\|\dot{q}_{ij}\|^{2}q_{ij}.\nonumber
\end{align*}
\end{prop}
\begin{proof}
See Appendix~\ref{sec:PfFDM}
\end{proof}
These equations are derived directly on a nonlinear manifold without any simplification. The dynamics of the payload, flexible cables, and quadrotors are considered explicitly, and they avoid singularities and complexities associated to local coordinates.

%%%%%%%%%%%%%%%%%%%%%%%%%%%%%%%%%%%%%%%%%%%%%%%%%%%%%%%%%%%%%%%
%%%%%%%%%%%%%%%%%%%%%%%%%%%%%%%%%%%%%%%%%%%%%%%%%%%%%%%%%%%%%%%
\section{CONTROL SYSTEM DESIGN FOR SIMPLIFIED DYNAMIC MODEL}

\subsection{Control Problem Formulation}

Let $x_{0_{d}}\in\Re^{3}$ be the desired position of the payload. The desired attitude of the payload is considered as $R_{0_d}=I_{3\times 3}$, and the desired direction of links is aligned along the vertical direction. The corresponding location of the $i$-th quadrotor at this desired configuration is given by
\begin{align}
x_{i_d}=x_{0_d}+\rho_{i}-\sum_{a=1}^{n_{i}}{l_{ia}e_{3}}.
\end{align}
We wish to design control forces $f_{i}$ and control moments $M_{i}$ of quadrotors such that this desired configuration becomes asymptotically stable.

\subsection{Simplified Dynamic Model}

Control forces for each quadrotor is given by $-f_{i}R_{i}e_{3}$ for the given equations of motion \refeqn{EOMM1}, \refeqn{EOMM2}, \refeqn{EOMM3}, \refeqn{EOMM4}. As such, the quadrotor dynamics is underactuated. The total thrust magnitude of each quadrotor can be arbitrary chosen, but the direction of the thrust vector is always along the third body fixed axis, represented by $R_ie_3$. But, the rotational attitude dynamics of the quadrotors are fully actuated, and they are not affected by the translational dynamics of the quadrotors or the dynamics of links. 

Based on these observations, in this section, we simplify the model by replacing the $-f_{i}R_{i}e_{3}$ term by a fictitious control input $u_{i}\in\Re^{3}$, and design an expression for $u$ to asymptotically stabilize the desired equilibrium. In another words, we assume that the attitude of the quadrotor can be instantaneously changed. The effects of the attitude dynamics are studied at the next section.

\subsection{Linear Control System}

The control system for the simplified dynamic model is developed based on the linearized equations of motion. At the desired equilibrium, the position and the attitude of the payload are given by $x_{0_d}$ and $R_{0}^{*}=I_{3}$, respectively, where the superscript $*$ denotes the value of a variable at the desired equilibrium throughout this paper. Also, we have $q_{ij}^{*}=e_{3}$ and $R_{i}^{*}=I_{3}$. In this equilibrium configuration, the control input for the $i$-th quadrotor is
\begin{align}
u_{i}^{*}=-f_{i}^{*}R_{i}^{*}e_{3},
\end{align}
where the total thrust is $f_{i}^{*}=(M_{iT}+\frac{m_{0}}{n})g$.

The variation of $x_{0}$ is given by
\begin{align}\label{eqn:xlin}
\delta x_{0}=x_{0}-x_{0_{d}},
\end{align}
and the variation of the attitude of the payload is defined as
\begin{align*}
\delta R_0 = R_0^* \hat\eta_0 = \hat\eta_0,
\end{align*}
for $\eta_0\in\Re^3$. The variation of $q_{ij}$ can be written as
\begin{align}\label{eqn:qlin}
\delta q_{ij}=\xi_{ij}\times e_{3},
\end{align}
where $\xi_{ij}\in\Re^{3}$ with $\xi_{ij}\cdot e_{3}=0$. The variation of $\omega_{ij}$ is given by $\delta \omega_{ij}\in\Re^{3}$ with $\delta\omega_{ij}\cdot e_{3}=0$. Therefore, the third element of each of $\xi_{ij}$ and $\delta\omega_{ij}$ for any equilibrium configuration is zero, and they are omitted in the following linearized equations. The state vector of the linearized equation is composed of $C^{T}\xi_{ij}\in\Re^{2}$, where $C=[e_{1},\;e_{2}]\in\Re^{3\times 2}$. The variation of the control input $\delta u_{i}\in\Re^{3\times 1}$, is given as $\delta u_i = u_i-u_i^*$.

\begin{prop}\label{prop:stability1}
The linearized equations of the simplified dynamic model are given by 
\begin{align}\label{eqn:EOMLin}
\Mb\ddot \xb  + \Gb\xb = \Bb \delta u,
\end{align}
where the state vector $\xb\in\Re^{D_{\xb}}$ with $D_{\xb}=6+2\sum_{i=1}^{n}n_{i}$ is given by
\begin{align*}
\xb=\begin{bmatrix}
\delta{x}_{0},\eta_{0},C^{T}{\xi}_{1j},C^T{\xi}_{2j},\cdots,C^{T}{\xi}_{nj}\\
\end{bmatrix},
\end{align*}
and $\delta u=[\delta u_{1}^T,\; \delta u_{2}^T,\;\cdots,\;\delta u_{n}^T]^{T}\in\Re^{3n\times 1}$. The matrix $\Mb\in\Re^{D_{\xb}\times D_{\xb}}$ are defined as
\begin{align*}\Mb=
\begin{bmatrix}
M_{T}I_{3}&\Mb_{x_{0}\Omega_{0}}&\Mb_{x_{0}1}&\Mb_{x_{0}2}&\cdots&\Mb_{x_{0}n}\\
\Mb_{\Omega_{0} x_{0}}&\bar{J}_{0}&\Mb_{\Omega_{0}1}&\Mb_{\Omega_{0}2}&\cdots&\Mb_{\Omega_{0}n}\\
\Mb_{1 x_{0}}&\Mb_{1\Omega_{0}}&\Mb_{qq1}&0&\cdots&0\\
\Mb_{2 x_{0}}&\Mb_{2\Omega_{0}}&0&\Mb_{qq2}&\cdots&0\\
\vdots&\vdots&\vdots&\vdots&\vdots&\vdots\\
\Mb_{n x_{0}}&\Mb_{n\Omega_{0}}&0&0&\cdots&\Mb_{qqn}
\end{bmatrix},
\end{align*}
where the sub-matrices are defined as
\begin{gather}
\Mb_{x_{0}\Omega_{0}}=-\sum_{i=1}^{n}{M_{iT}\hat{\rho}_{i}};\; \Mb_{\Omega_{0} x_{0}}=\Mb_{x_{0}\Omega_{0}}^{T},\nonumber\\
\Mb_{x_{0}i}=[M_{0i1}l_{i1}\hat{e}_{3}C,\; M_{0i2}l_{i2}\hat{e}_{3}C,\; \cdots,\;M_{0in_{i}}l_{in_{i}}\hat{e}_{3}C],\nonumber\\
\Mb_{\Omega_{0}i}=[M_{0i1}l_{i1}\hat{\rho}_{i}C,\; M_{0i2}l_{i2}\hat{\rho}_{i}C,\; \cdots,\; M_{0in_{i}}l_{in_{i}}\hat{\rho}_{i}C],\nonumber\\
\Mb_{ix_{0}}=-[M_{0i1}C^{T}\hat{e}_{3},\; M_{0i2}C^{T}\hat{e}_{3},\; \cdots,\; M_{0in_{i}}C^{T}\hat{e}_{3}],\\
\Mb_{i\Omega_{0}}=[M_{0i1}C^{T}\hat{e}_{3}\hat{\rho}_{i},\; M_{0i2}C^{T}\hat{e}_{3}\hat{\rho}_{i},\;\cdots,\; M_{0in_{i}}C^{T}\hat{e}_{3}\hat{\rho}_{i}],
\end{gather}
and the sub-matrix $\Mb_{qqi}\in\Re^{2n_i\times 2n_i}$ is given by
\begin{align}\Mb_{qqi}=
\begin{bmatrix}
M_{i11}l_{i1}I_{2}&M_{i12}l_{i2}I_{2}&\cdots&M_{i1n_{i}}l_{in_{i}}I_{2}\\
M_{i21}l_{i1}I_{2}&M_{i22}l_{i2}I_{2}&\cdots&M_{i2n_{i}}l_{in_{i}}I_{2}\\
\vdots&\vdots&&\vdots\\
M_{in_{i}1}l_{i1}I_{2}&M_{in_{i}2}l_{i2}I_{2}&\cdots&M_{in_{i}n_{i}}l_{in_{i}}I_{2}
\end{bmatrix}.
\end{align}
The matrix $\Gb\in\Re^{D_{\xb}\times D_{\xb}}$ is defined as
\begin{align*}
\Gb=\begin{bmatrix}
0&0&0&0&0&0\\
0&\Gb_{\Omega_{0}\Omega_{0}}&0&0&0&0\\
0&0&\Gb_{1}&0&0&0\\
0&0&0&\Gb_{2}&0&0\\
\vdots&\vdots&\vdots&\vdots&\vdots&\vdots\\
0&0&0&0&0&\Gb_{n}
\end{bmatrix},
\end{align*}
where $\Gb_{\Omega_{0}\Omega_{0}}=\sum_{i=1}^{n}\frac{m_{0}}{n}g\hat{\rho}_{i}\hat{e}_{3}$ and the sub-matrices $\Gb_{i}\in\Re^{2n_{i}\times 2n_{i}}$ are
\begin{align*}
\Gb_{i}= \diag[(-M_{iT}-\frac{m_{0}}{n}+M_{0ij})ge_{3}I_{2}].
\end{align*}
The matrix $\Bb\in\Re^{D_{\xb}\times 3n}$ is given by 
\begin{align*}
\Bb=\begin{bmatrix}
I_{3}&I_{3}&\cdots&I_{3}\\
\hat{\rho}_{1}&\hat{\rho}_{2}&\cdots&\hat{\rho}_{n}\\
\Bb_{\Bb}&0&0&0\\
0&\Bb_{\Bb}&0&0\\
\vdots&\vdots&\vdots&\vdots\\
0&0&0&\Bb_{\Bb}
\end{bmatrix},
\end{align*}
where $\Bb_{\Bb}=-[C^{T}\hat{e}_{3},\; C^{T}\hat{e}_{3},\; \cdots,\; C^{T}\hat{e}_{3}]^{T}$.
\end{prop}
\begin{proof}
See Appendix~\ref{sec:P1stability}
\end{proof}
We present the following PD-type control system for the linearized dynamics
\begin{align}
\delta u_{i}%=&-k_{x_{i}}\delta x_{0}-k_{\dot{x}_{i}}\delta \dot{x}_{0}-k_{\eta_{{0}_{i}}}\eta_{0}-k_{\Omega_{{0}_{i}}}\delta\Omega_{0}\nonumber\\
%&-\sum_{j=1}^{n_{i}}k_{q_{ij}}C^{T}(e_{3}\times q_{ij})-k_{\omega_{ij}}C^{T}(\delta \omega_{ij})\nonumber\\
=&-K_{x_{i}}\xb-K_{\dot{x}_{i}}\dot{\xb},
\end{align}
for controller gains $K_{x_{i}},K_{\dot{x}_{i}}\in\Re^{3\times D_{\xb}}$. Provided that \refeqn{EOMLin} is controllable, we can choose the combined controller gains $K_{x}=[K_{x_{1}}^T,\,\ldots\;K_{x_{n}}^T]^{T}$, $K_{\dot{x}}=[K_{\dot{x}_{1}}^T,\,\ldots K_{\dot{x}_{n}}^T]^{T}\in\Re^{3n\times D_{\xb}}$ such that the equilibrium is asymptotically stable for the linearized equation \refeqn{EOMLin}. 
%Then, the equilibrium becomes asymptotically stable for the nonlinear Euler-Lagrange equation~\cite{Kha96}.
%%%%%%%%%%%%%%%%%%%%%%%%%%%%%%%%%%%%%%%%%%%%%%%%%%%%%%%%%%%%%%%
%%%%%%%%%%%%%%%%%%%%%%%%%%%%%%%%%%%%%%%%%%%%%%%%%%%%%%%%%%%%%%%

\section{CONTROL SYSTEM DESIGN FOR THE FULL DYNAMIC MODEL}
The control system designed at the previous section is based on a simplifying assumption that each quadrotor can generates a thrust along any direction. %However, a quadrotor is under-actuated since the direction of the total thrust is always parallel to its third body-fixed axis, while the magnitude of the total thrust can be arbitrarily changed. This can be directly observed from the expression of the total thrust, $u_{i}=-f_{i}R_{i}e_{3}$, where $f_{i}$ is the total thrust magnitude, and $R_{i}e_{3}$ corresponds to the direction of the third body-fixed axis. Conversely, the rotational attitude dynamics is fully actuated by the arbitrary control moment $M_{i}$. 
In the full dynamic model, the direction of the thrust for each quadrotor is parallel to its third body-fixed axis always. In this section, the attitude of each quadrotor is controlled such that the third body-fixed axis becomes parallel to the direction of the ideal control force designed in the previous section. The central idea is that the attitude $R_{i}$ of the quadrotor is controlled such that its total thrust direction $-R_{i}e_{3}$, corresponding to the third body-fixed axis, asymptotically follows the direction of the fictitious control input $u_{i}$. By choosing the total thrust magnitude properly, we can guarantee asymptotical stability for the full dynamic model. 

Let $A_{i}\in\Re^{3}$ be the ideal total thrust of the $i$-th quadrotor that asymptotically stabilize the desired equilibrium. Therefor, we have
\begin{align}\label{eqn:Ai}
A_{i}=u_{i}^{*}+\delta u_{i}=-K_{x_{i}}\xb-K_{\dot{x}_{i}}\dot{\xb}+u_{i}^{*},
\end{align} 
where $f_{i}^{*}$ and $u_{i}^{*}$ are the total thrust and control input of each quadrotor at its equilibrium respectively. 

From the desired direction of the third body-fixed axis of the $i$-th quadrotor, namely $b_{3_{i}}\in\Sph^2$, is given by
\begin{align}
b_{3_{i}}=-\frac{A_{i}}{\|A_{i}\|}.
\end{align}
This provides a two-dimensional constraint on the three dimensional desired attitude of each quadrotor, such that there remains one degree of freedom. To resolve it, the desired direction of the first body-fixed axis $b_{1_{i}}(t)\in\Sph^2$ is introduced as a smooth function of time. Due to the fact that the first body-fixed axis is normal to the third body-fixed axis, it is impossible to follow an arbitrary command $b_{1_{i}}(t)$ exactly. Instead, its projection onto the plane normal to $b_{3_{i}}$ is followed, and the desired direction of the second body-fixed axis is chosen to constitute an orthonormal frame~\cite{LeeLeoAJC13}. More explicitly, the desired attitude of the $i$-th quadrotor is given by
\begin{align}
R_{i_{c}}=\begin{bmatrix}
-\frac{(\hat{b}_{3_{i}})^{2}b_{1_{i}}}{\|(\hat{b}_{3_{i}})^{2}b_{1_{i}}\|} & \frac{\hat{b}_{3_{i}}b_{1_{i}}}{\|\hat{b}_{3_{i}}b_{1_{i}}\|} & b_{3_{i}}\end{bmatrix},
\end{align}
which is guaranteed to be an element of $\so$. The desired angular velocity is obtained from the attitude kinematics equation, $\Omega_{i_{c}}=(R_{i_{c}}^{T}\dot{R}_{i_{c}})^\vee\in\Re^{3}$.
\begin{figure}
\centerline{
	\subfigure[Payload position ($x_0$:blue, $x_{0_{d}}$:red)]{
		\includegraphics[width=0.55\columnwidth]{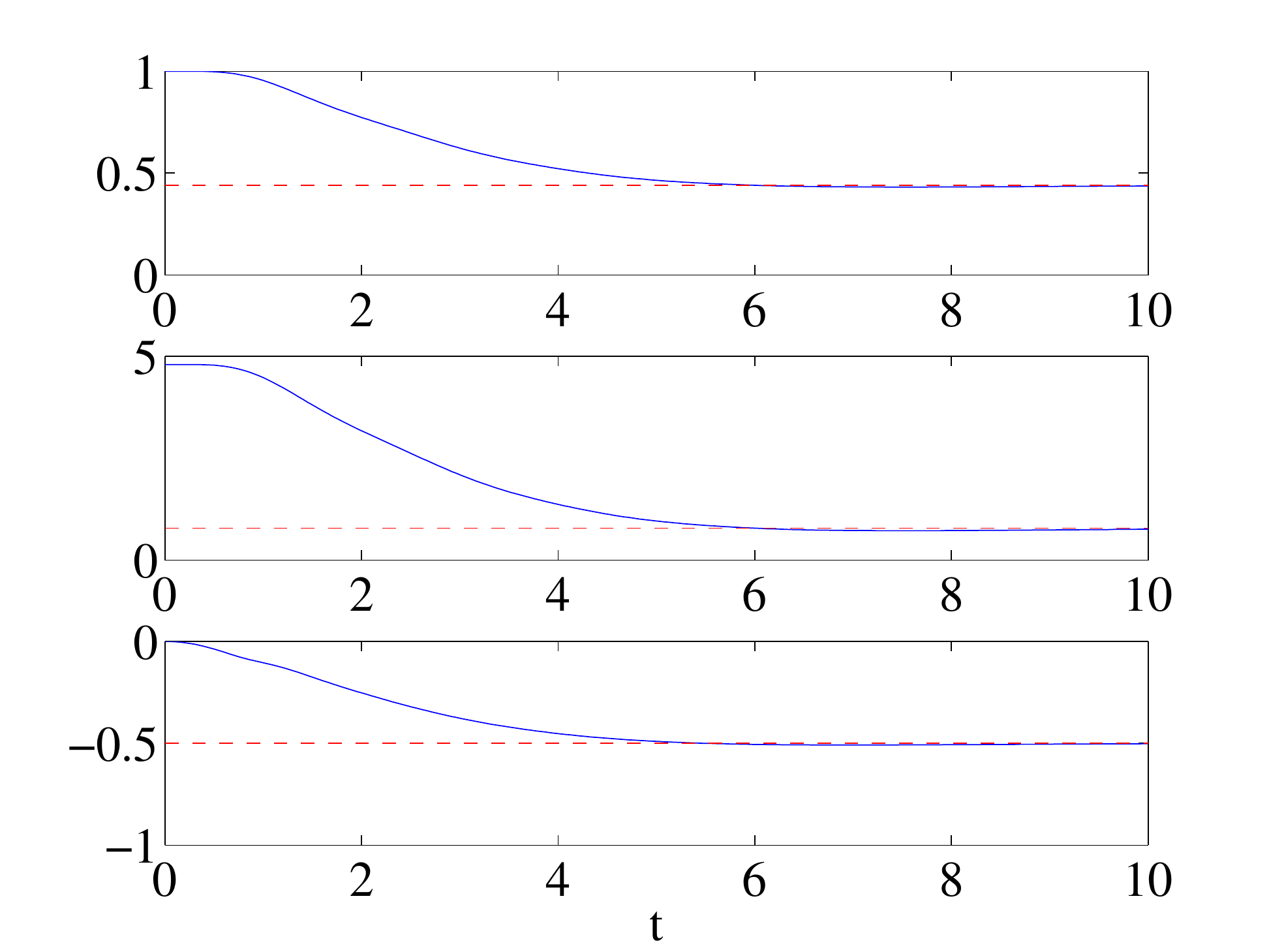}}
	\subfigure[Payload velocity ($v_0$:blue, $v_{0_{d}}$:red)]{
		\includegraphics[width=0.55\columnwidth]{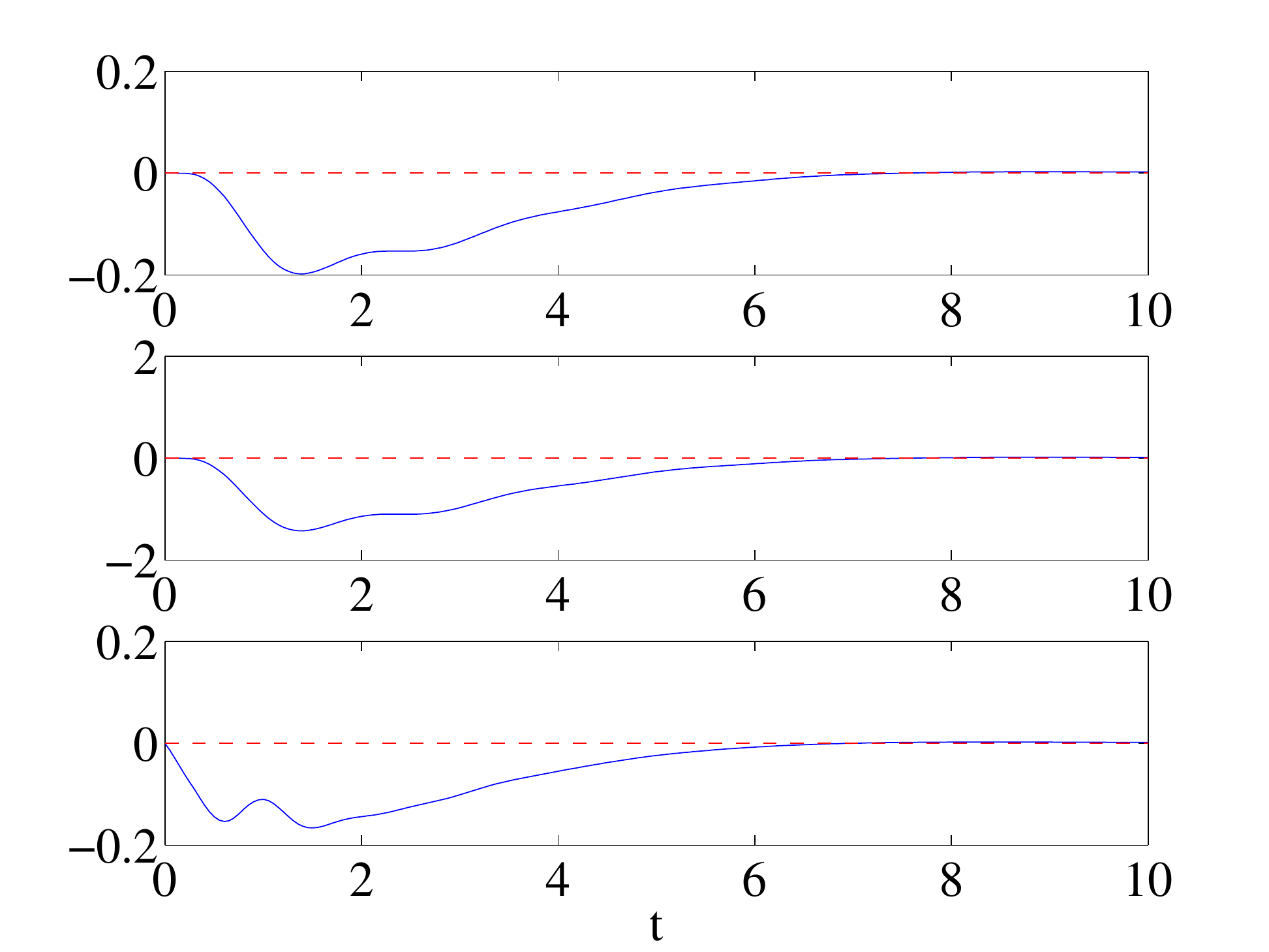}}
}
\centerline{
	\subfigure[Payload angular velocity $\Omega_{0}$]{
		\includegraphics[width=0.55\columnwidth]{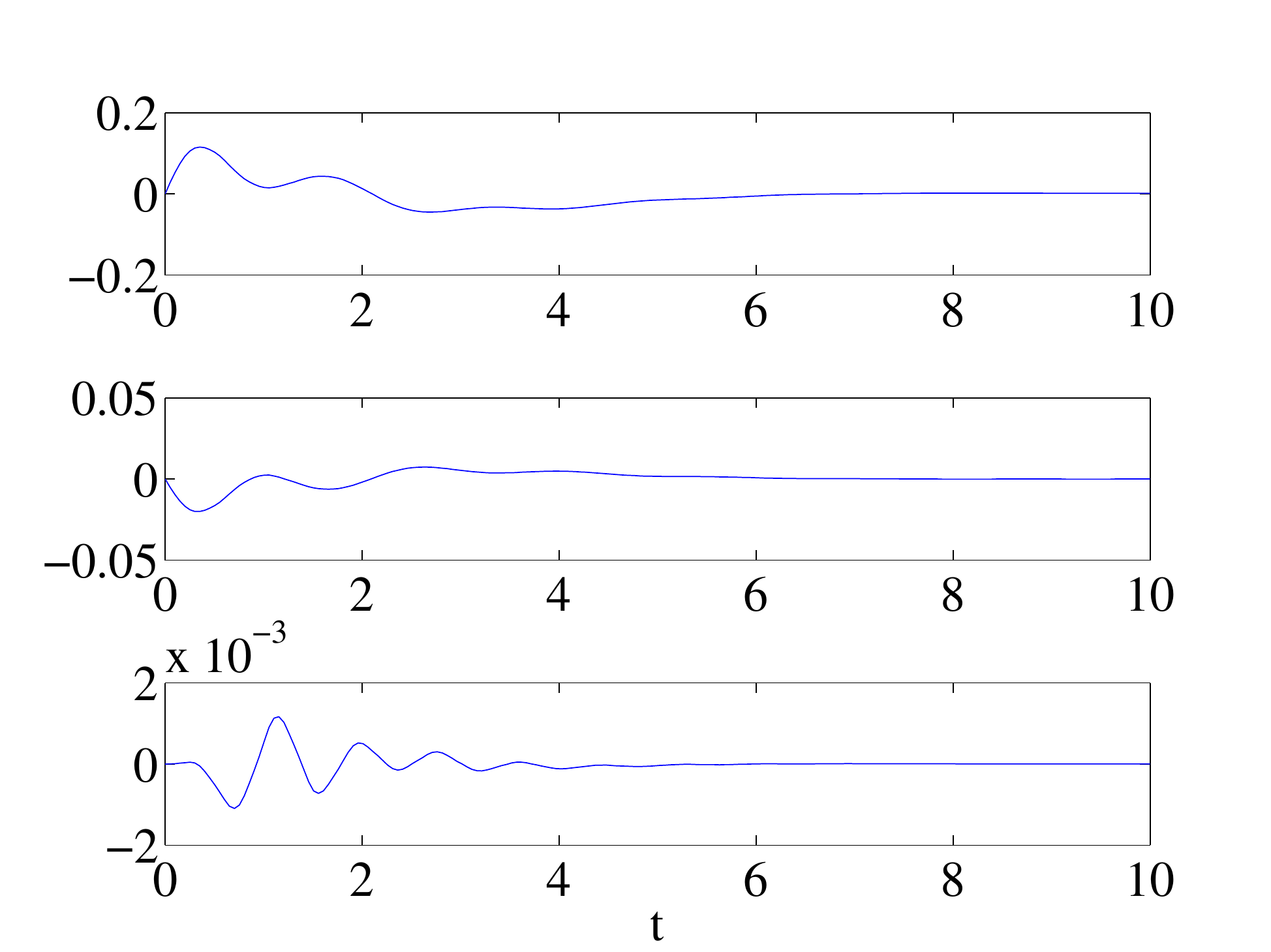}}
	\subfigure[Quadrotors angular velocity errors $e_{\Omega_{i}}$]{
		\includegraphics[width=0.55\columnwidth]{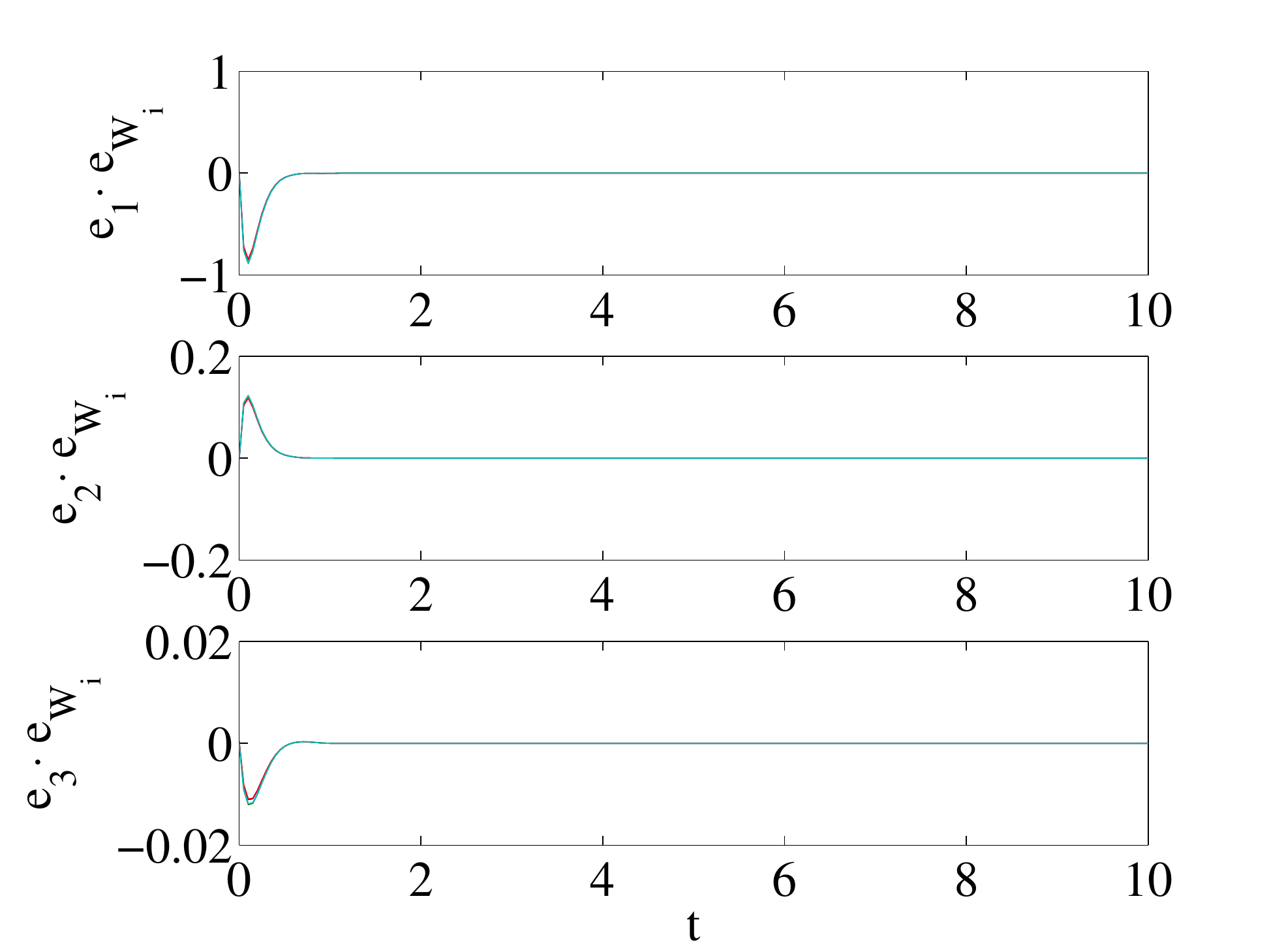}}
}
\centerline{
	\subfigure[Payload attitude error $\Psi_{0}$]{
		\includegraphics[width=0.55\columnwidth]{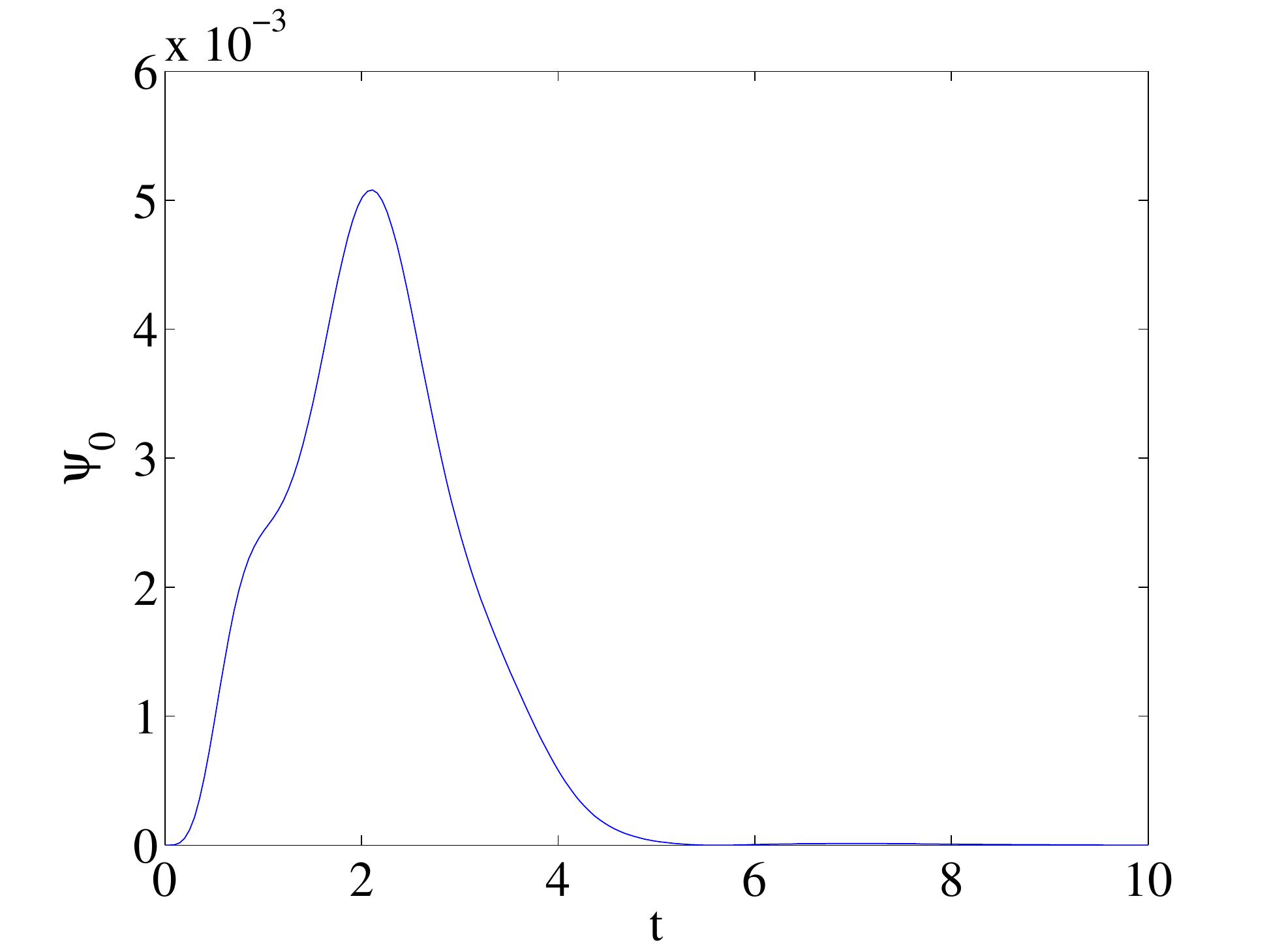}}
	\subfigure[Quadrotors attitude errors $\Psi_{i}$]{
		\includegraphics[width=0.55\columnwidth]{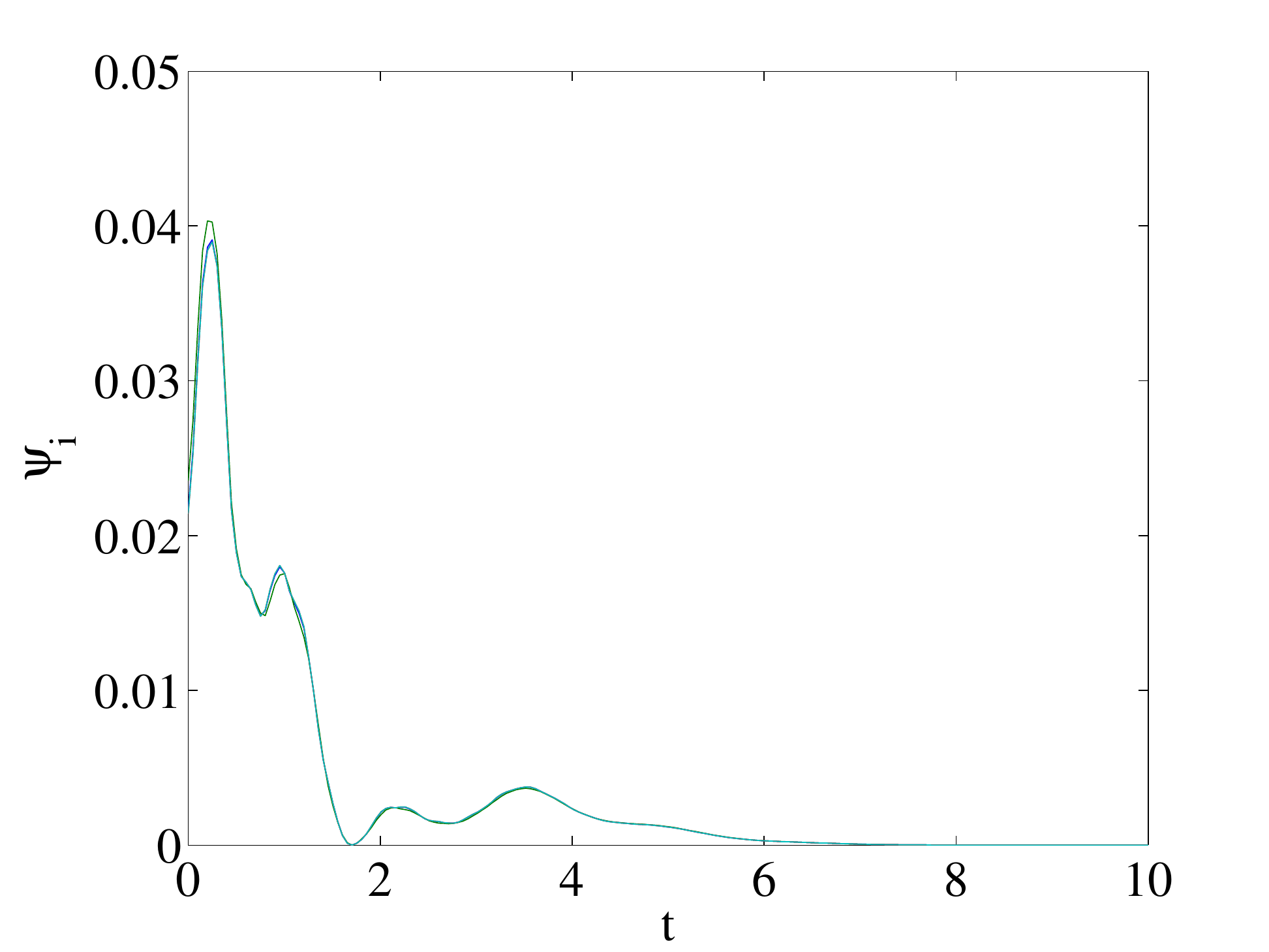}}
}
\centerline{
	\subfigure[Quadrotors total thrust inputs $f_{i}$]{
		\includegraphics[width=0.55\columnwidth]{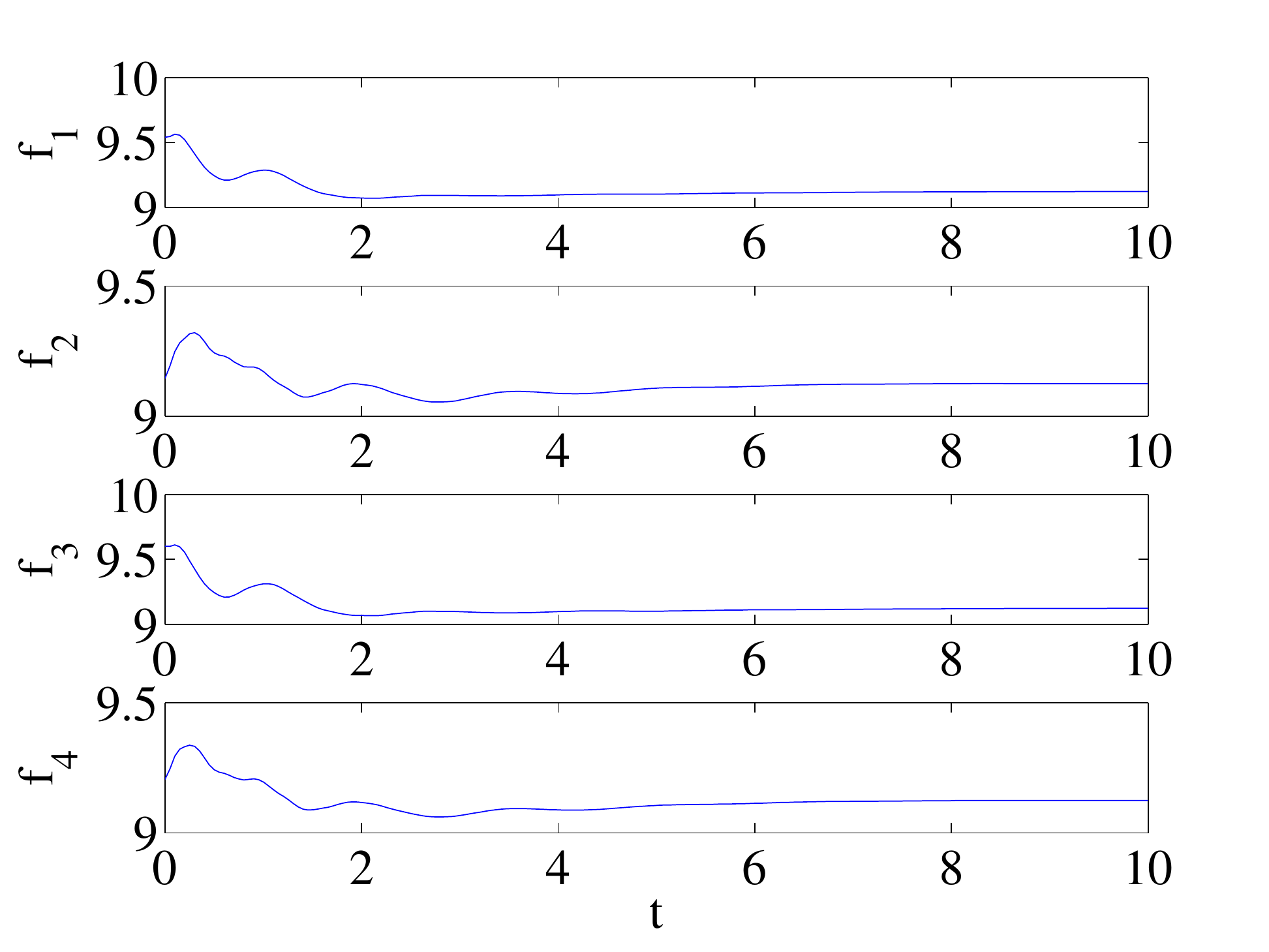}}
	\subfigure[Direction error $e_{q}$, and angular velocity error $e_{\omega}$ for the links]{
		\includegraphics[width=0.55\columnwidth]{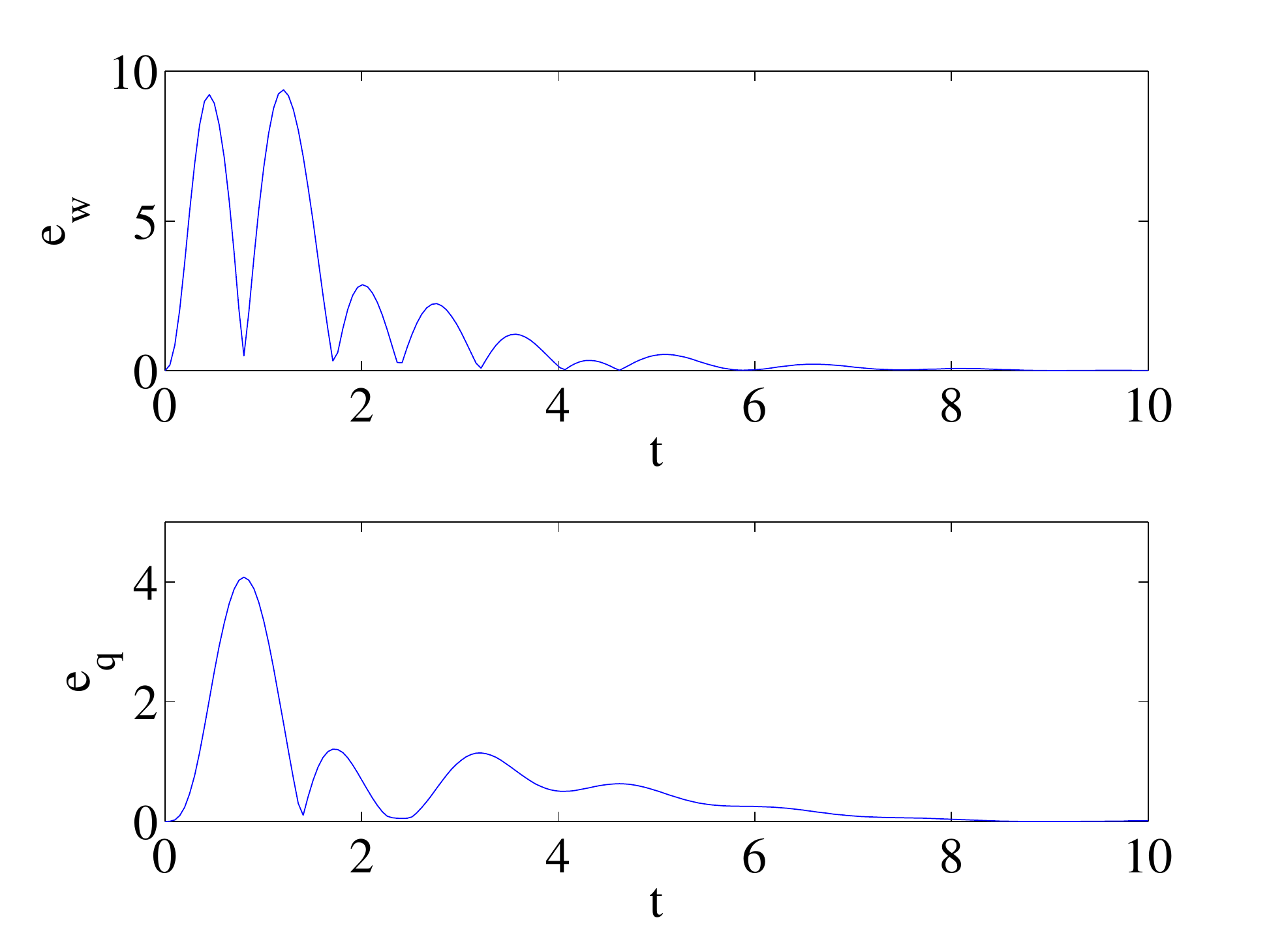}}
}
\caption{Stabilization of a rigid-body connected to multiple quadrotors}\label{fig:simresults1}
\end{figure}
Define the tracking error vectors for the attitude and the angular velocity of the $i$-th quadrotor as
\begin{align}
e_{R_{i}}=\frac{1}{2}(R_{i_{c}}^{T}R_{i}-R_{i}^{T}R_{i_{c}})^{\vee},\; e_{\Omega_{i}}=\Omega_{i}-R_{i}^{T}R_{i_{c}}\Omega_{i_{c}},
\end{align}
and a configuration error function on $\SO$ as follows
\begin{align}
\Psi_{i}= \frac{1}{2}\trs[I- R_{{i}_c}^T R_{i}].
\end{align}
The thrust magnitude is chosen as the length of $u_{i}$, projected on to $-R_{i}e_{3}$, and the control moment is chosen as a tracking controller on $\SO$:
\begin{align}
f_{i}=&-A_{i}\cdot R_{i}e_{3},\label{eqn:fi}\\
M_{i}=&-k_R e_{R_{i}} -k_\Omega e_{\Omega_{i}}\nonumber\\
&+(R_{i}^TR_{c_i}\Omega_{c_{i}})^\wedge J_{i} R_{i}^T R_{c_i} \Omega_{c_i} + J_{i} R_{i}^T R_{c_i}\dot\Omega_{c_i},\label{eqn:Mi}
\end{align}
where $k_{R}$ and $k_{\Omega}$ are positive constants.

Stability of the corresponding controlled systems for the full dynamic model can be studied by showing the the error due to the discrepancy between the desired direction $b_{3_{i}}$ and the actual direction $R_{i}e_{3}$. This stability is shown via a Lyapunov analysis.
%For both cases, the structures of the control systems are identical, and here we use singular perturbation for simplicity.
\begin{prop}\label{prop:stability}
Consider the full dynamic model defined by \refeqn{EOMM1}, \refeqn{EOMM2}, \refeqn{EOMM3}, \refeqn{EOMM4}. For the command $x_{0_{d}}$ and the desired direction of the first body-fixed axis $b_{1_{i}}$, control inputs for quadrotors are designed as \refeqn{fi} and \refeqn{Mi}. Then, the equilibrium of zero tracking errors for $e_{x_{0}},\; \dot{e}_{x_{0}},\; e_{R_0},\;e_{\Omega_{0}},\; e_{q_{ij}},\; e_{\omega_{ij}},\; e_{R_i},\;e_{\Omega_{i}}$, is exponentially stable.
\end{prop}
\begin{proof}
See Appendix \ref{sec:Pstabilityddd}
\end{proof}

%%%%%%%%%%%%%%%%%%%%%%%%%%%%%%%%%%%%%%%%%%%%%%%%%%%%%%%%%%%%%%%
%%%%%%%%%%%%%%%%%%%%%%%%%%%%%%%%%%%%%%%%%%%%%%%%%%%%%%%%%%%%%%%
\section{NUMERICAL EXAMPLE}

We demonstrate the desirable properties of the proposed control system with numerical examples. Two cases are presented. At the first case, a payload is transported to a desired position from the ground. The second case considers stabilization of a payload with large initial attitude errors.

\subsection{Stabilization of the Rigid Body}
Consider four quadrotors $(n=4)$ connected via flexible cables  to a rigid body payload. Initial conditions are chosen as
\begin{gather*}
x_{0}(0)=[1.0,\; 4.8,\; 0.0]^{T}\,\mathrm{m},\; v_{0}(0)=0_{3\times 1},\\
q_{ij}(0)=e_{3},\; \omega_{ij}(0)=0_{3\times 1},\; R_{i}(0)=I_{3\times 3},\; \Omega_{i}(0)=0_{3\times 1}\\
R_{0}(0)=I_{3\times3},\; \Omega_{0}=0_{3\times 1}.
\end{gather*}
The desired position of the payload is chosen as
\begin{align}
x_{0_{d}}(t)=[0.44,\; 0.78,\; -0.5]^{T}\,\mathrm{m}.
\end{align}
The mass properties of quadrotors are chosen as
\begin{gather}
m_{i}=0.755\,\mathrm{kg},\nonumber\\ 
J_{i}=\diag[0.557,\; 0.557,\; 1.05]\times 10^{-2} \mathrm{kgm^2}.
\end{gather}
The payload is a box with mass $m_{0}=0.5\,\mathrm{kg}$, and its length, width, and height are $0.6$, $0.8$, and $0.2\,\mathrm{m}$, respectively. Each cable connecting the rigid body to the $i$-th quadrotor is considered to be $n_{i}=5$ rigid links. All the links have the same mass of $m_{ij}=0.01\,\mathrm{kg}$ and length of $l_{ij}=0.15\,\mathrm{m}$. Each cable is attached to the following points of the payload
\begin{gather*}
\rho_{1}=[0.3,\; -0.4,\; -0.1]^T\,\mathrm{m},\; \rho_{2}=[0.3,\; 0.4,\; -0.1]^T\,\mathrm{m},\\
\rho_{3}=[-0.3,\; -0.4,\; -0.1]^T\,\mathrm{m},\; \rho_{4}=[-0.3,\; 0.4,\; -0.1]^T\,\mathrm{m}.
\end{gather*}

Numerical simulation results are presented at Figure \ref{fig:simresults1}, which shows the position and velocity of the payload, and its tracking errors. We have also presented the link direction error defined as
\begin{align*}
e_{q}=\sum_{i=1}^{m}\sum_{j=1}^{n_{i}}{\|q_{ij}-e_{3}\|}.%,\; e_{\omega}=\sum_{i=1}^{m}\sum_{j=1}^{n_{i}}{\|\omega_{ij}\|}.
\end{align*}

\begin{figure}
\centerline{
	\subfigure[3D perspective]{
		\includegraphics[width=0.9\columnwidth]{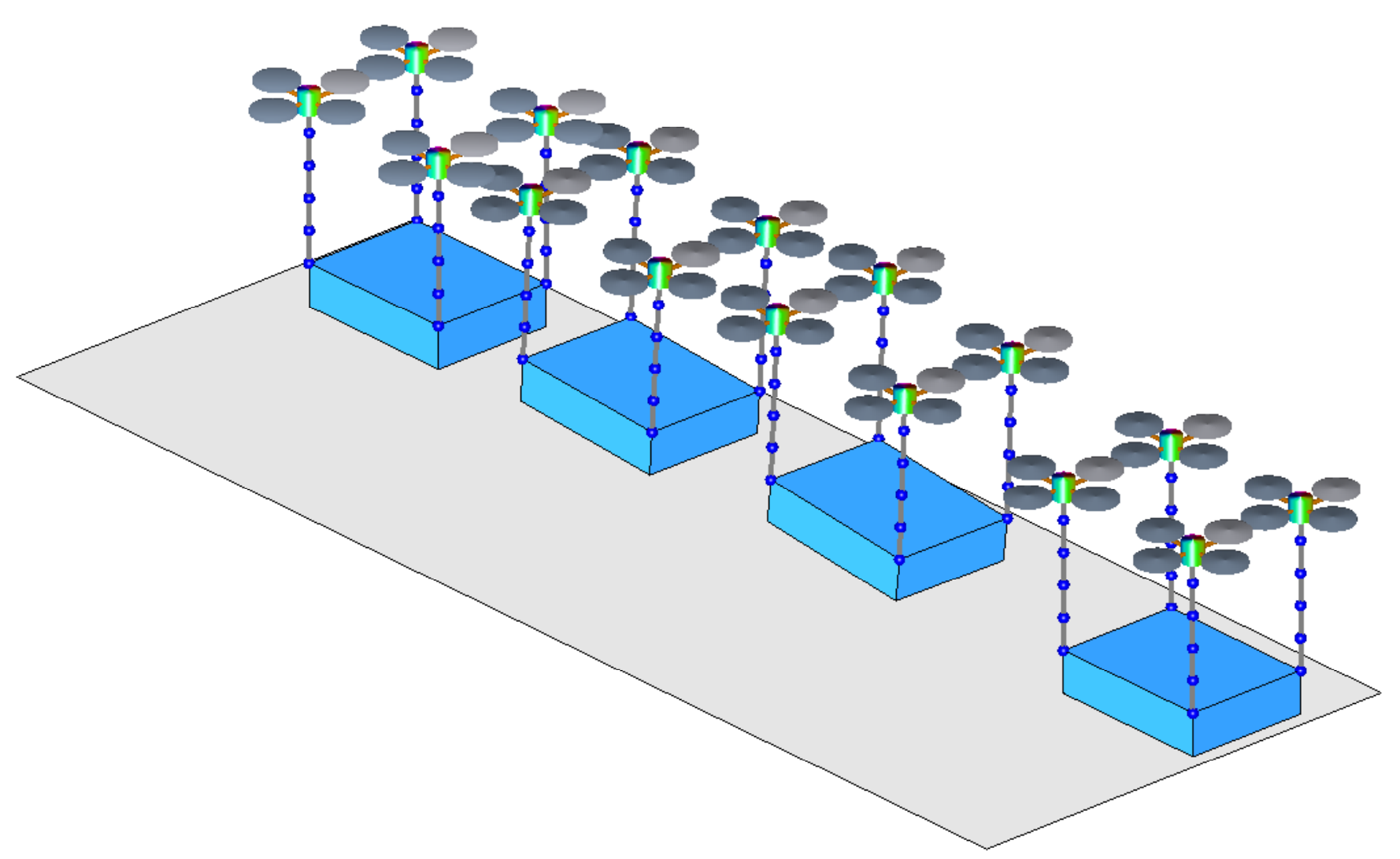}}
}
\centerline{
	\subfigure[Side view]{
		\includegraphics[width=0.9\columnwidth]{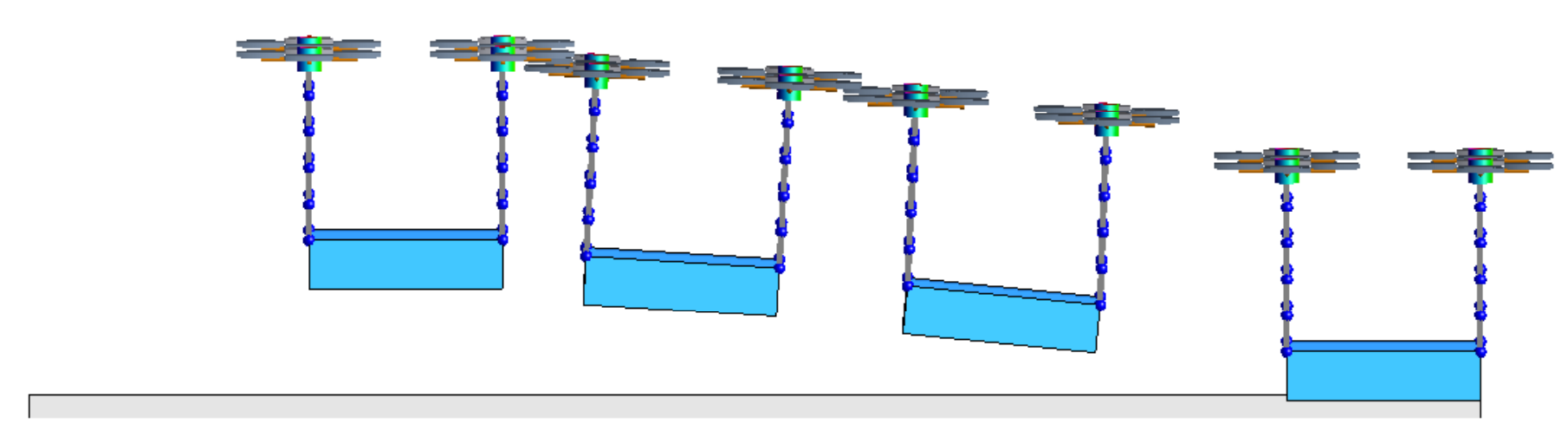}}
}
\centerline{
	\subfigure[Top view]{
		\includegraphics[width=0.9\columnwidth]{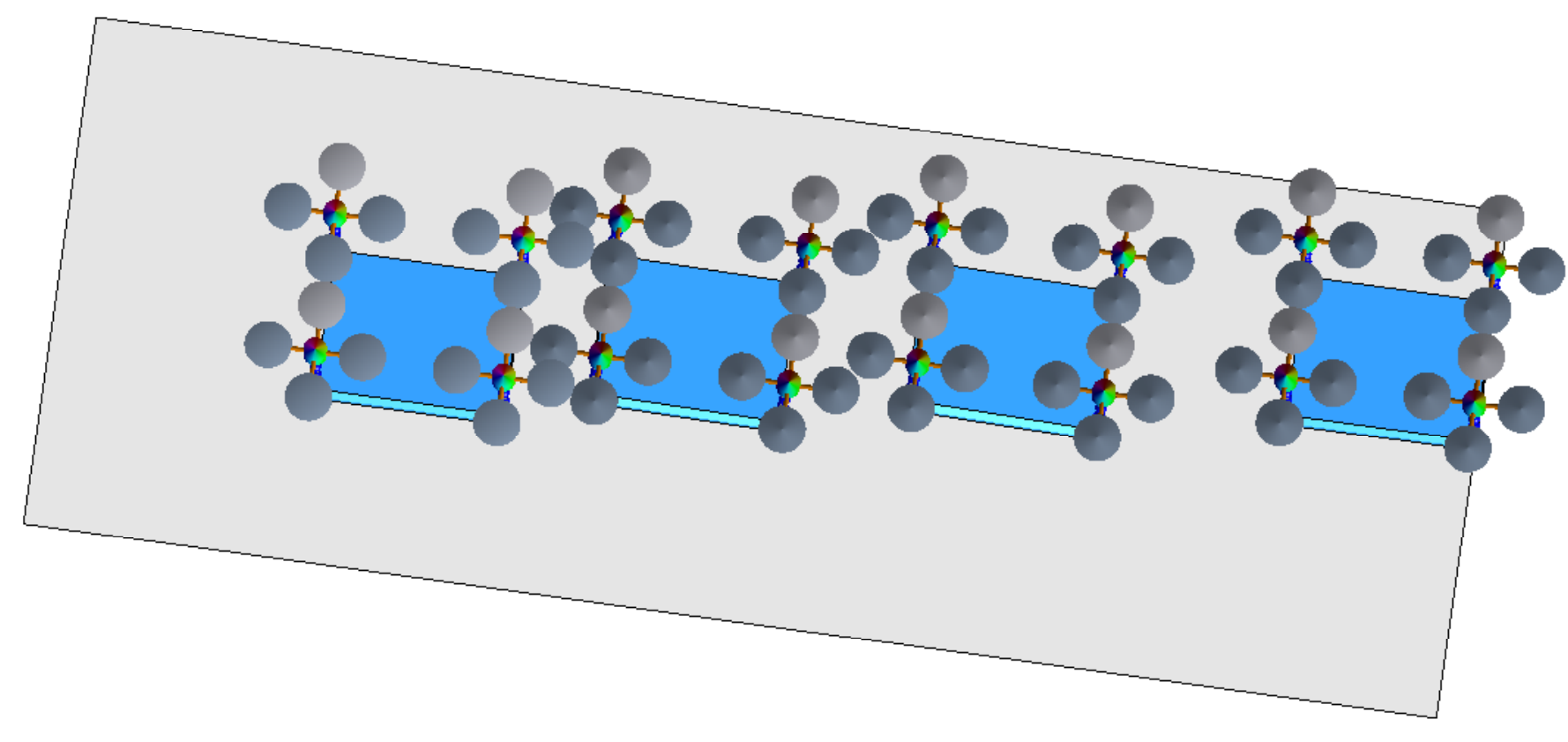}}
}
\caption{Snapshots of controlled maneuver}\label{fig:simresults1snap}
\end{figure}

\begin{figure}
\centerline{\hspace{-0.5cm}
	\subfigure[Payload position $x_0$:blue, $x_{0_{d}}$:red]{
		\includegraphics[width=0.5\columnwidth]{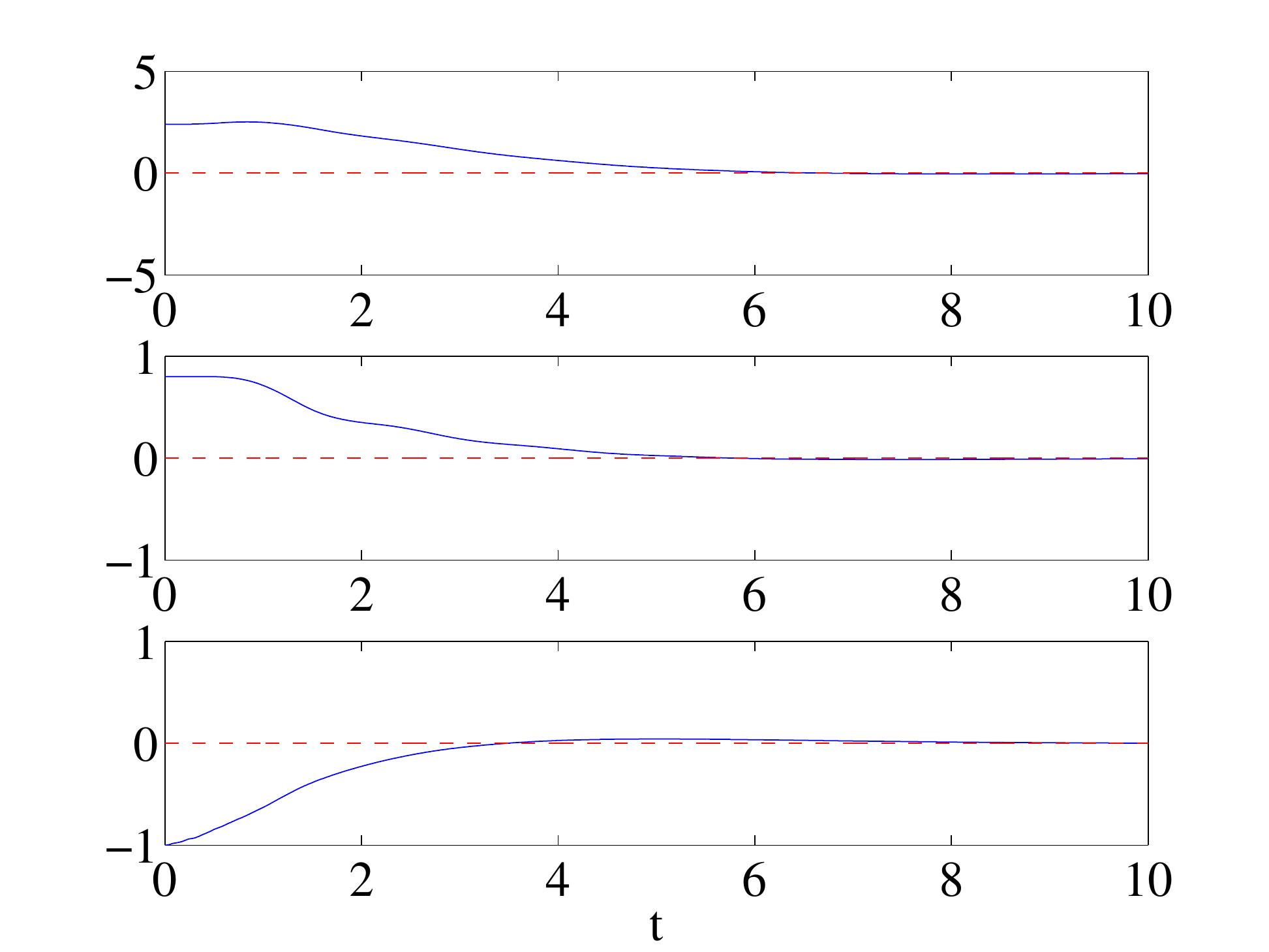}\label{fig:case2x0}}
	\subfigure[Payload velocity $v_0$:blue, $v_{0_{d}}$:red]{
		\includegraphics[width=0.5\columnwidth]{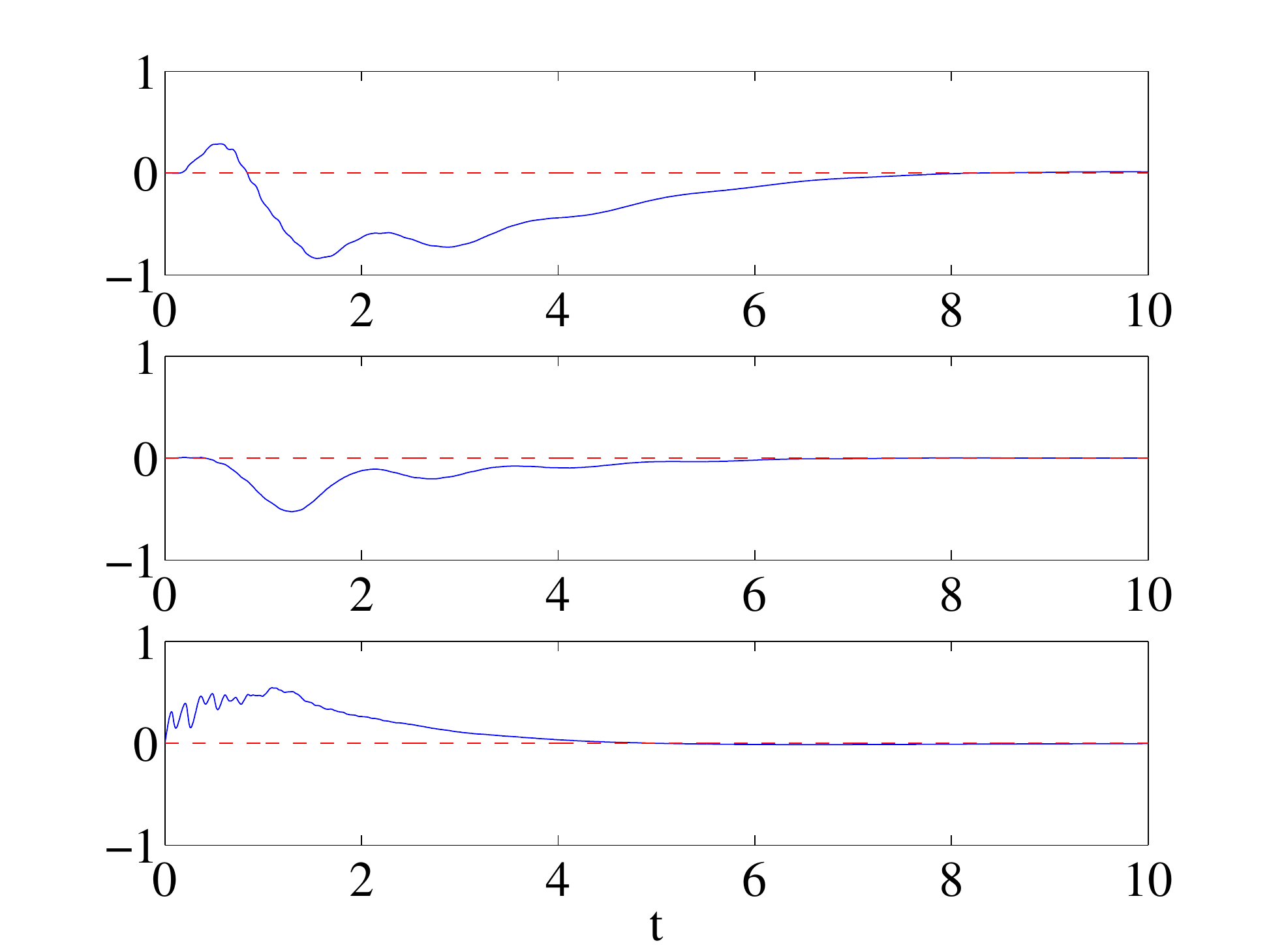}\label{fig:case2v0}}
}
\centerline{\hspace{-0.5cm}
	\subfigure[Payload angular velocity $\Omega_{0}$]{
		\includegraphics[width=0.5\columnwidth]{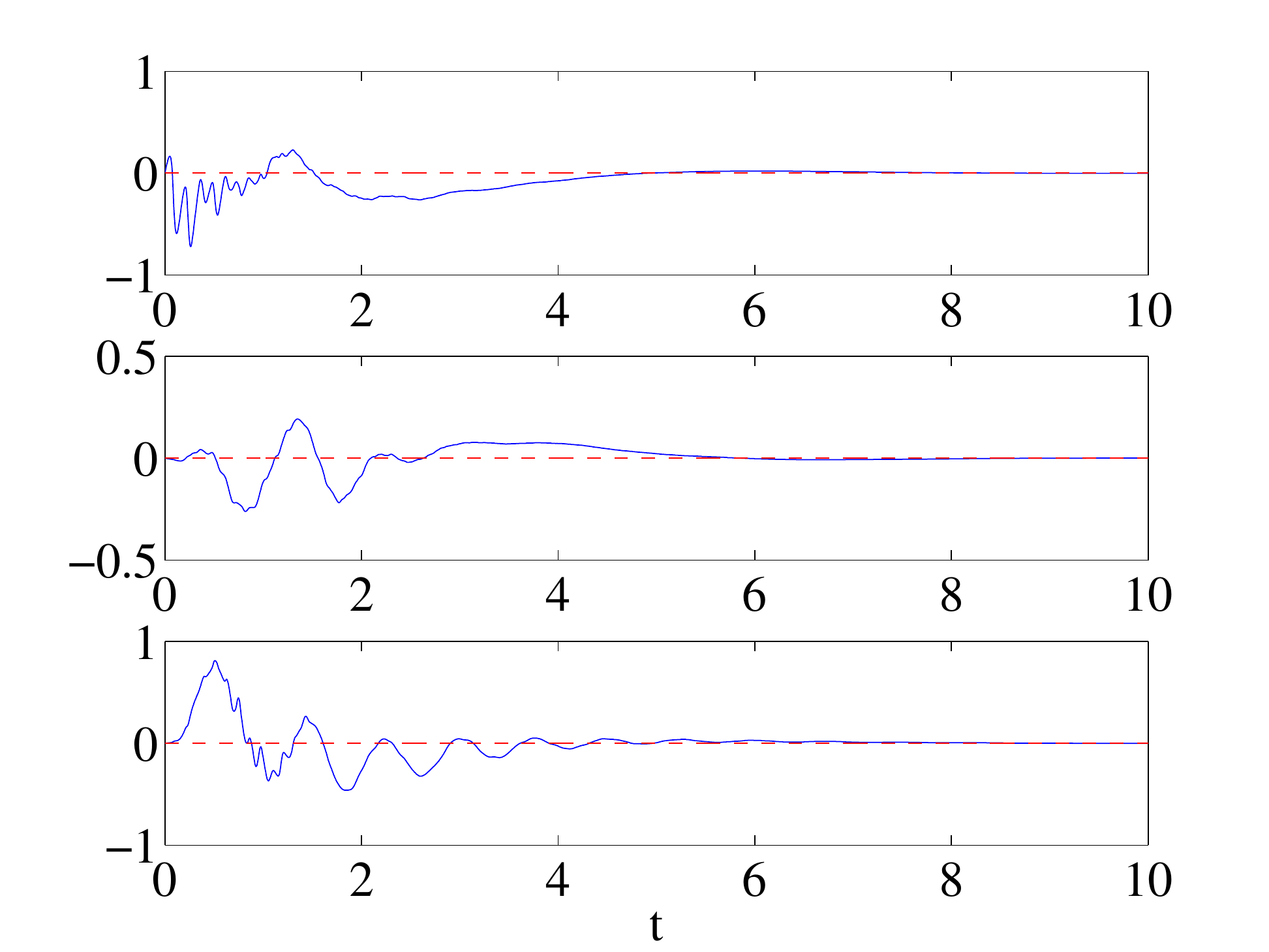}\label{fig:case2W0}}
	\subfigure[Quadrotors angular velocity errors $e_{\Omega_{i}}$]{
		\includegraphics[width=0.5\columnwidth]{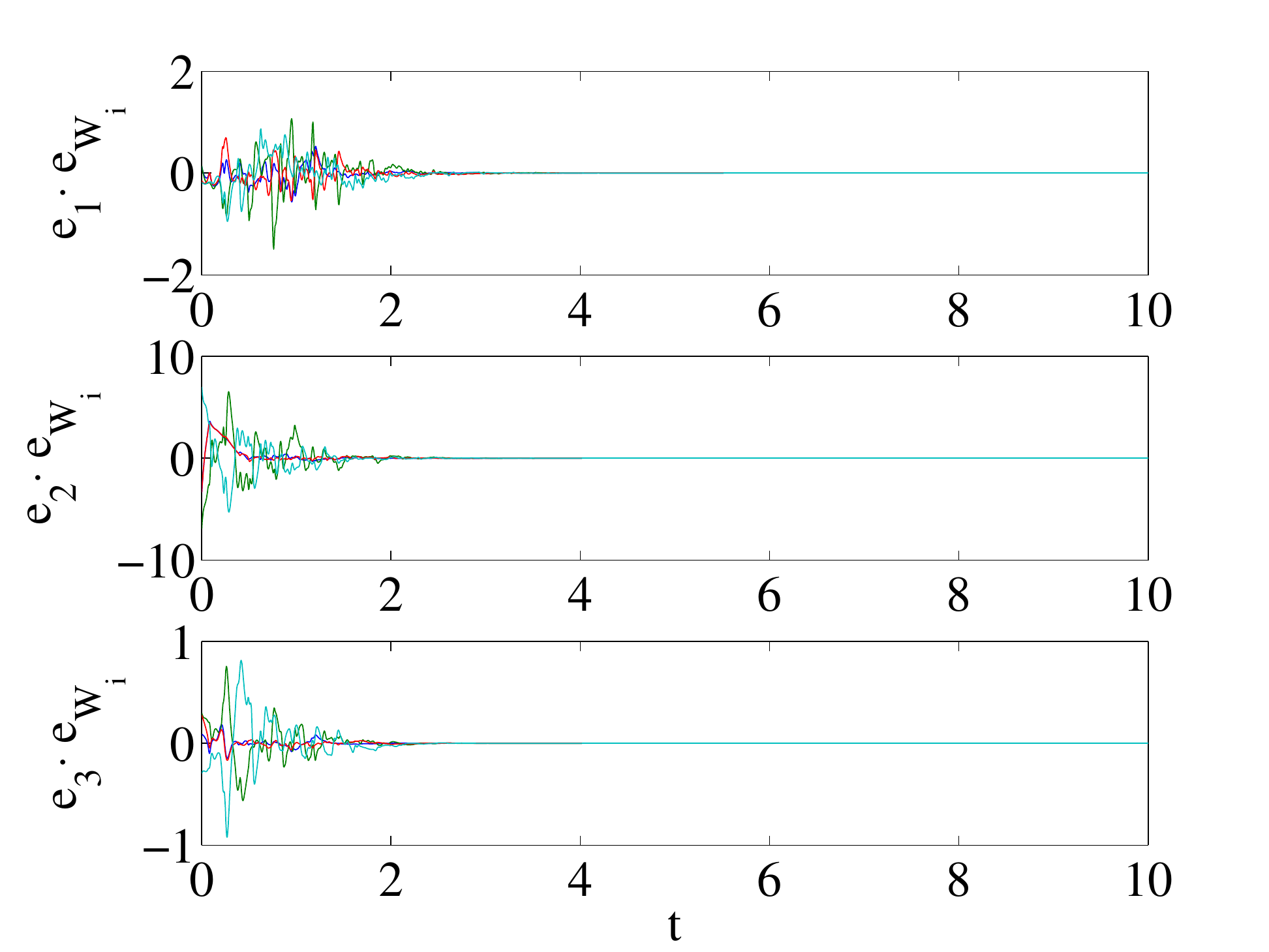}\label{fig:case2eW}}
}
\centerline{\hspace{-0.5cm}
	\subfigure[Payload attitude error $\psi_{0}$]{
		\includegraphics[width=0.5\columnwidth]{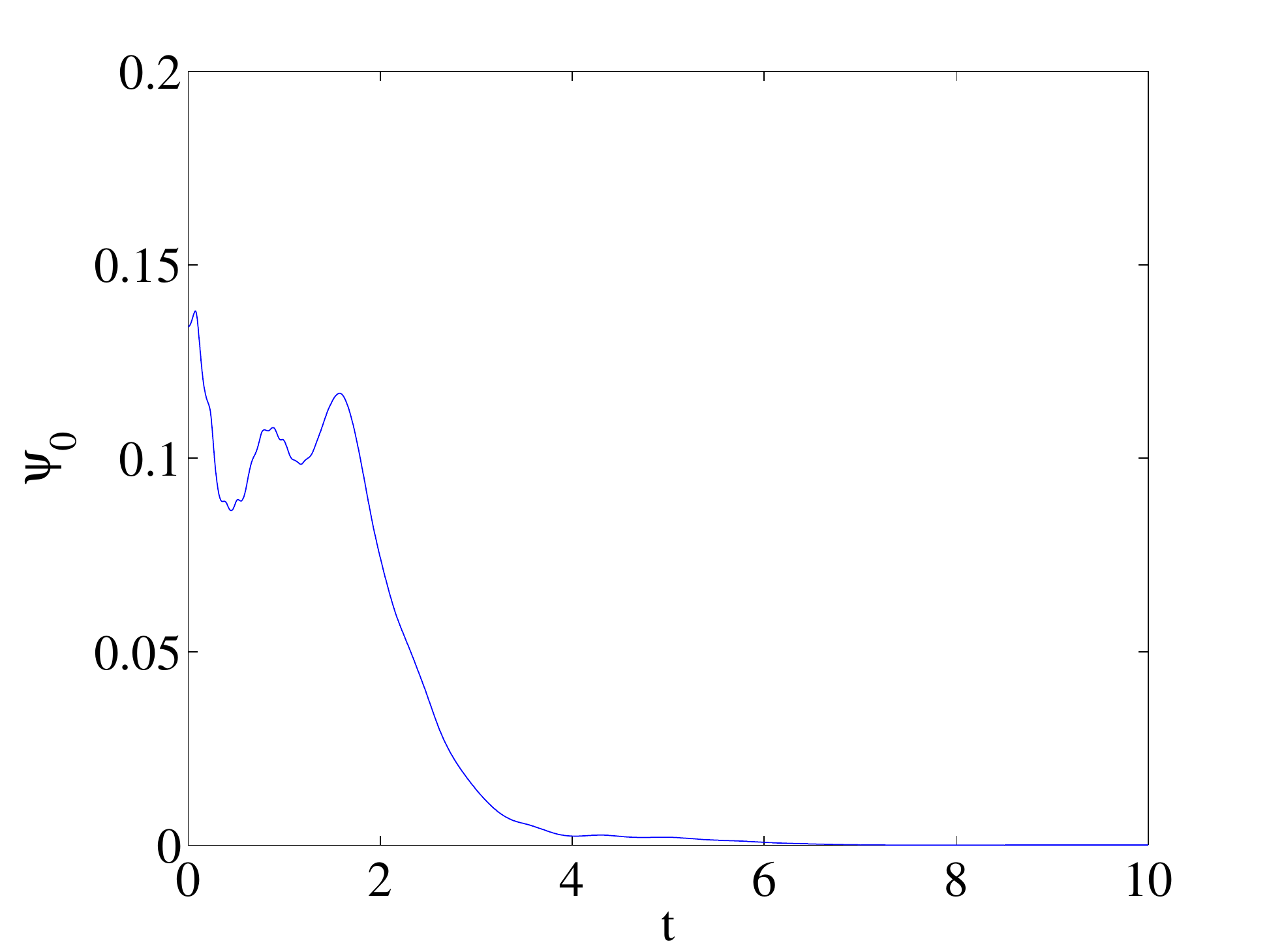}\label{fig:case2psi0}}
	\subfigure[Quadrotors attitude errors $\psi_{i}$]{
		\includegraphics[width=0.5\columnwidth]{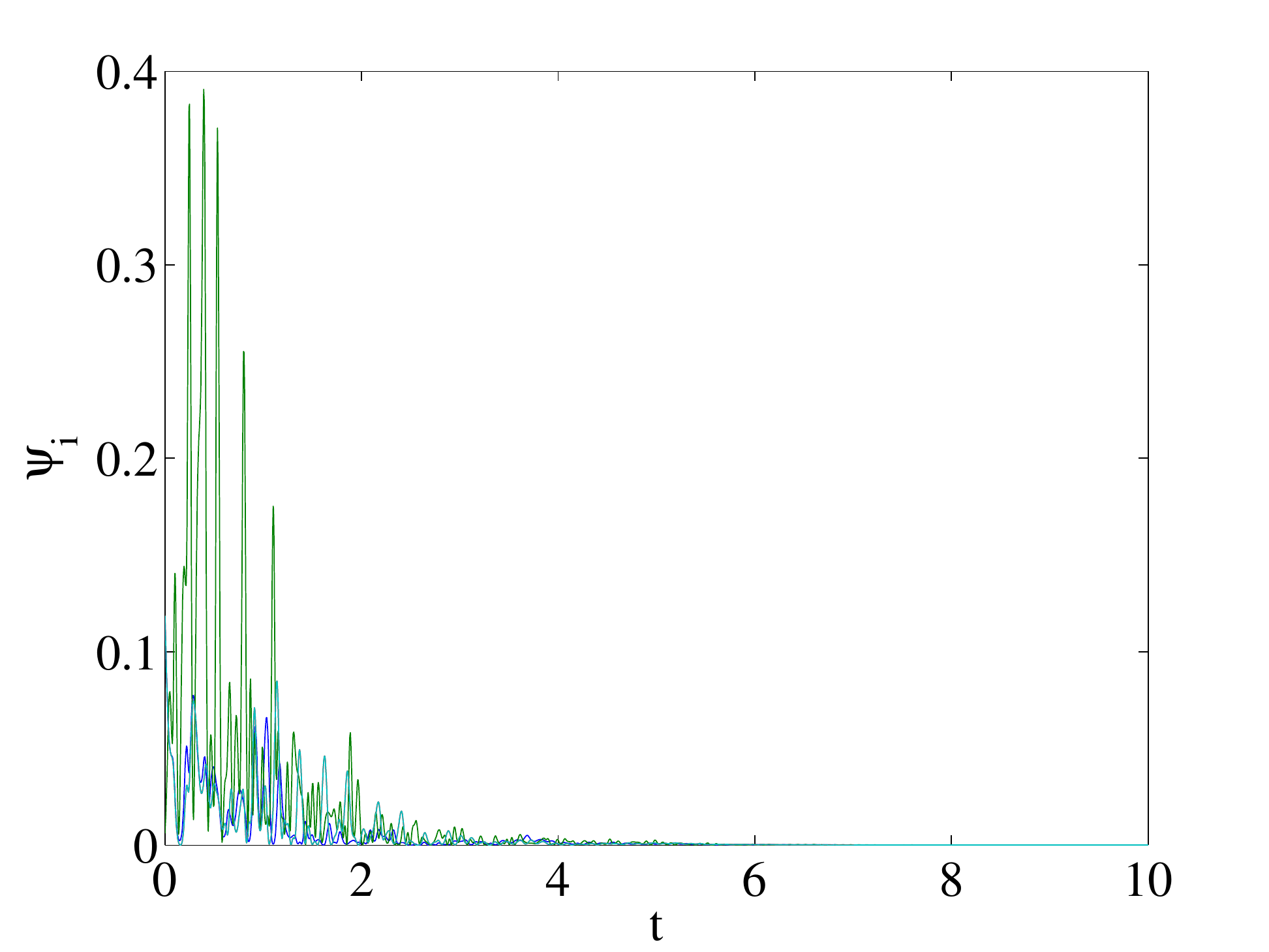}\label{fig:case2psii}}
}
\centerline{\hspace{-0.5cm}
	\subfigure[Quadrotors total thrust inputs $f_{i}$]{
		\includegraphics[width=0.5\columnwidth]{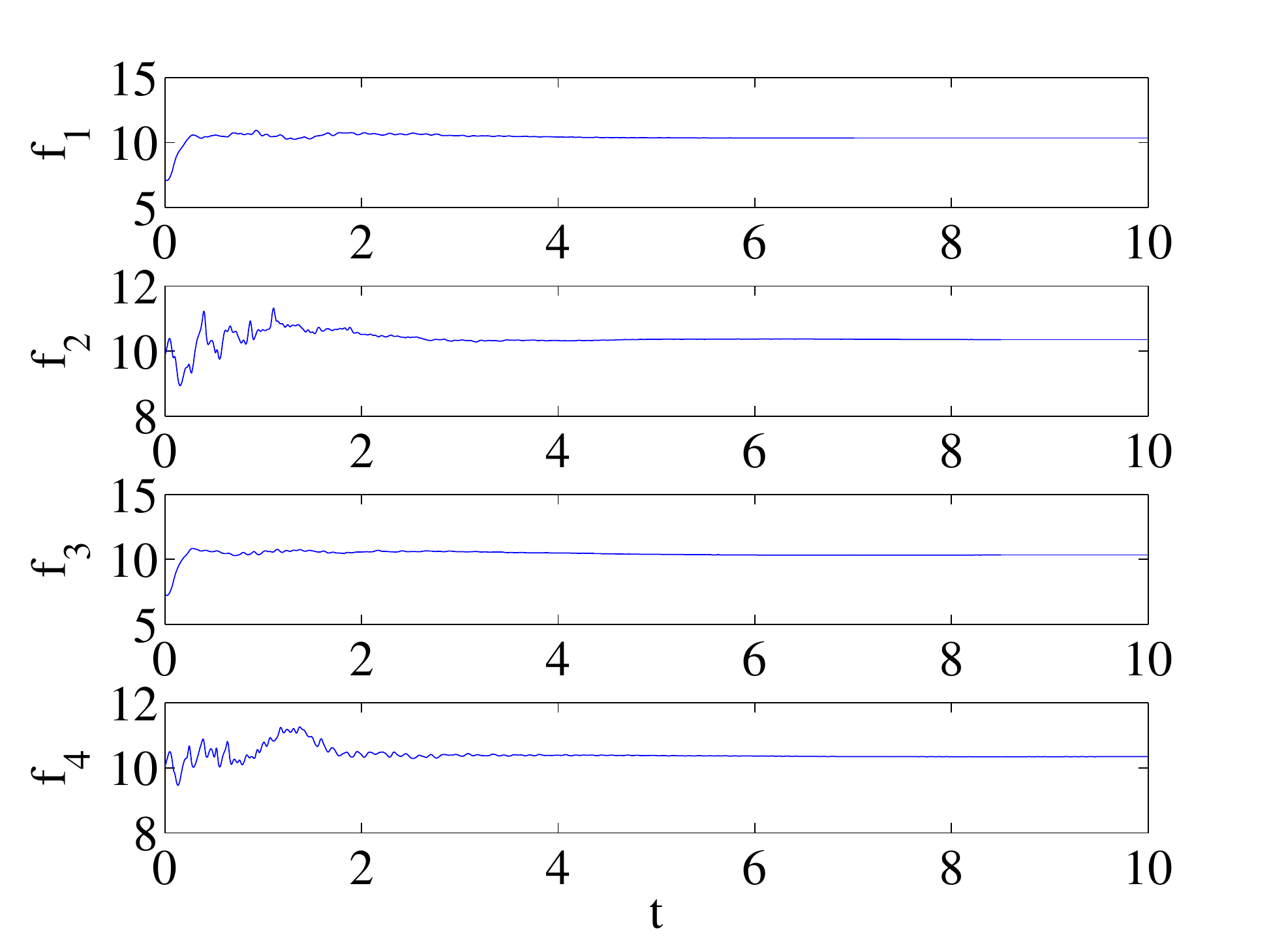}\label{fig:case2ui}}
	\subfigure[Direction error $e_{q}$, and angular velocity error $e_{\omega}$ for the links]{
		\includegraphics[width=0.5\columnwidth]{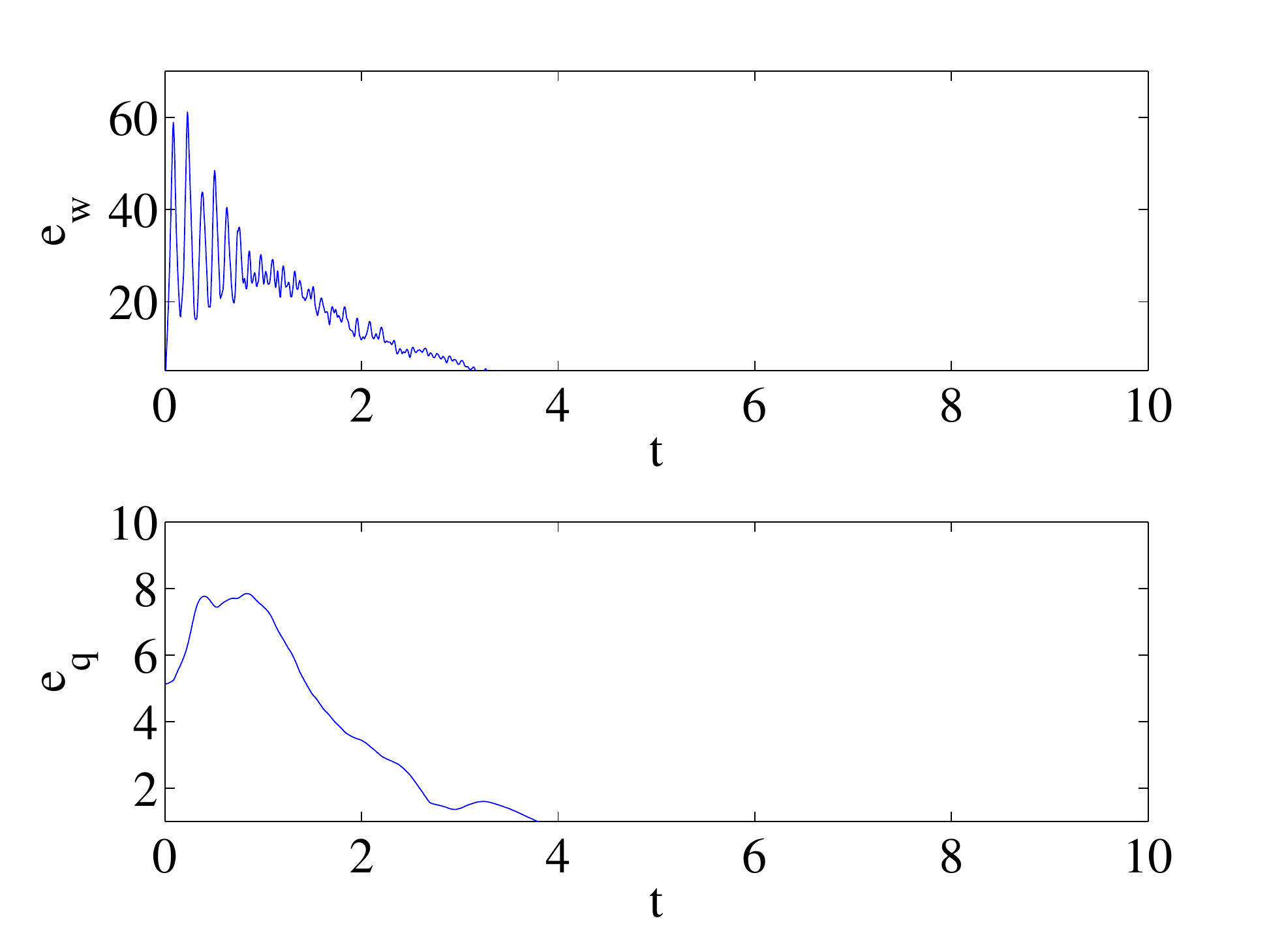}\label{fig:case2errors}}
}
\caption{Stabilization of a payload with multiple quadrotors connected with flexible cables.}\label{fig:simresults2}
\end{figure}
%%%%%%%%%%%%%%%%%%%%%%%%%%%%%%%%%%%%%%%%%
\subsection{Payload Stabilization with Large Initial Attitude Errors}
In the second case, we consider large initial errors for the attitude of the payload and quadrotors. Initially, the rigid body is tilted in its $b_{1}$ axis by $30$ degrees, and the initial direction of the links are chosen such that two cables are curved along the horizontal direction. The initial conditions are given by
\begin{gather*}
x_{0}(0)=[2.4,\; 0.8,\; -1.0]^{T},\; v_{0}(0)=0_{3\times 1},\\
\omega_{ij}(0)=0_{3\times 1},\; \Omega_{i}(0)=0_{3\times 1}\\
R_{0}(0)=R_{x}(30^\circ),\; \Omega_{0}=0_{3\times 1},
\end{gather*}
where $R_x(30^\circ)$ denotes the rotation about the first axis by $30^\circ$. The initial attitude of quadrotors are chosen as
\begin{gather*}
R_{1}(0)=R_{y}(-35^\circ),\; R_{2}(0)=I_{3\times 3},\\ 
R_{3}(0)=R_{y}(-35^\circ),\; R_{4}(0)=I_{3\times 3}.
\end{gather*}

The properties of quadrotors and cables are identical to the previous case. The payload mass is $m=1.0\,\mathrm{kg}$ , and its length, width, and height are $1.0$, $1.2$, and $0.2\,\mathrm{m}$, respectively.

\begin{figure}
\centerline{
	\subfigure[$t=0$ Sec.]{
		\includegraphics[width=0.3\columnwidth]{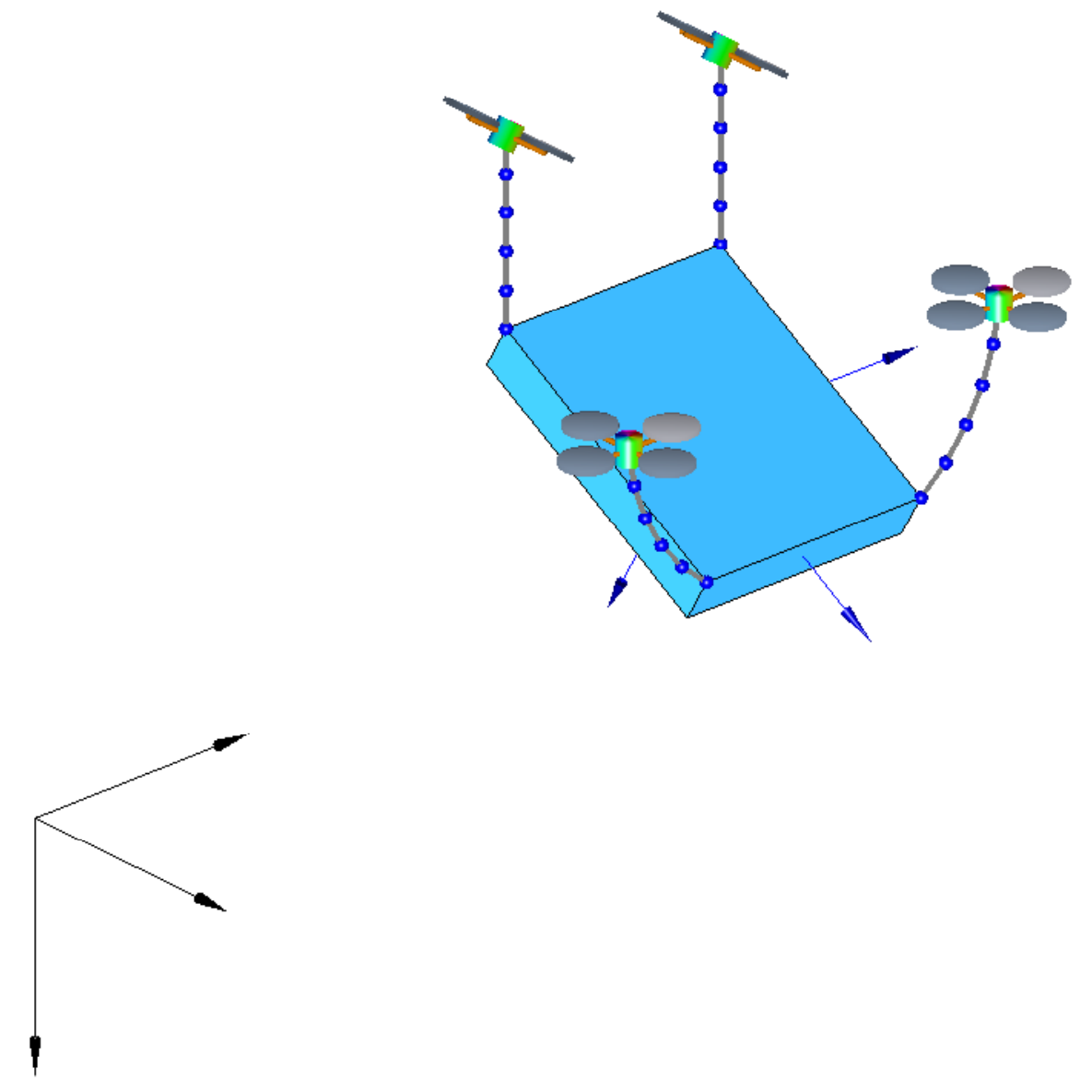}}
	\subfigure[$t=0.14$ Sec.]{
		\includegraphics[width=0.3\columnwidth]{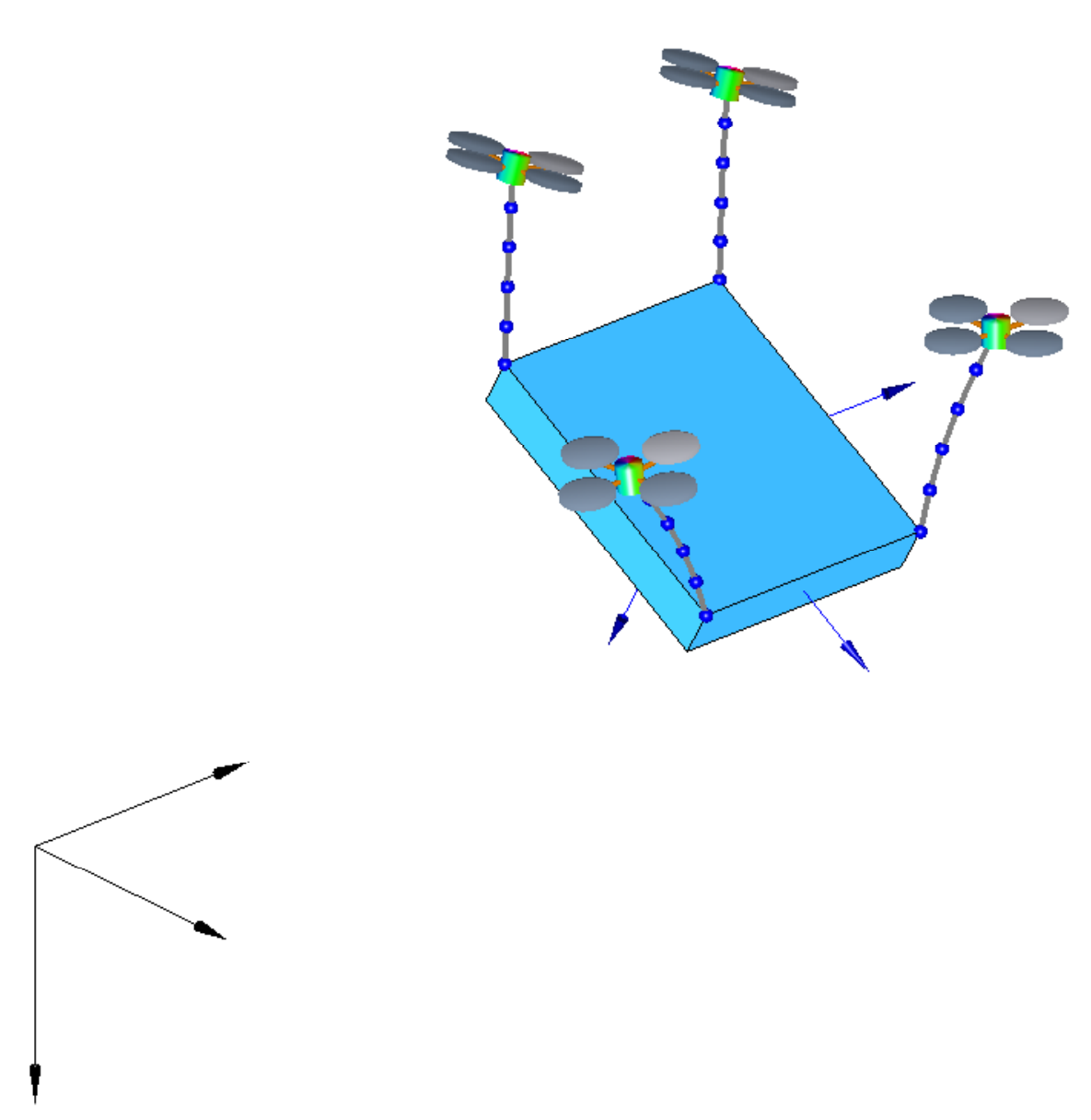}}
		\subfigure[$t=0.30$ Sec.]{
		\includegraphics[width=0.3\columnwidth]{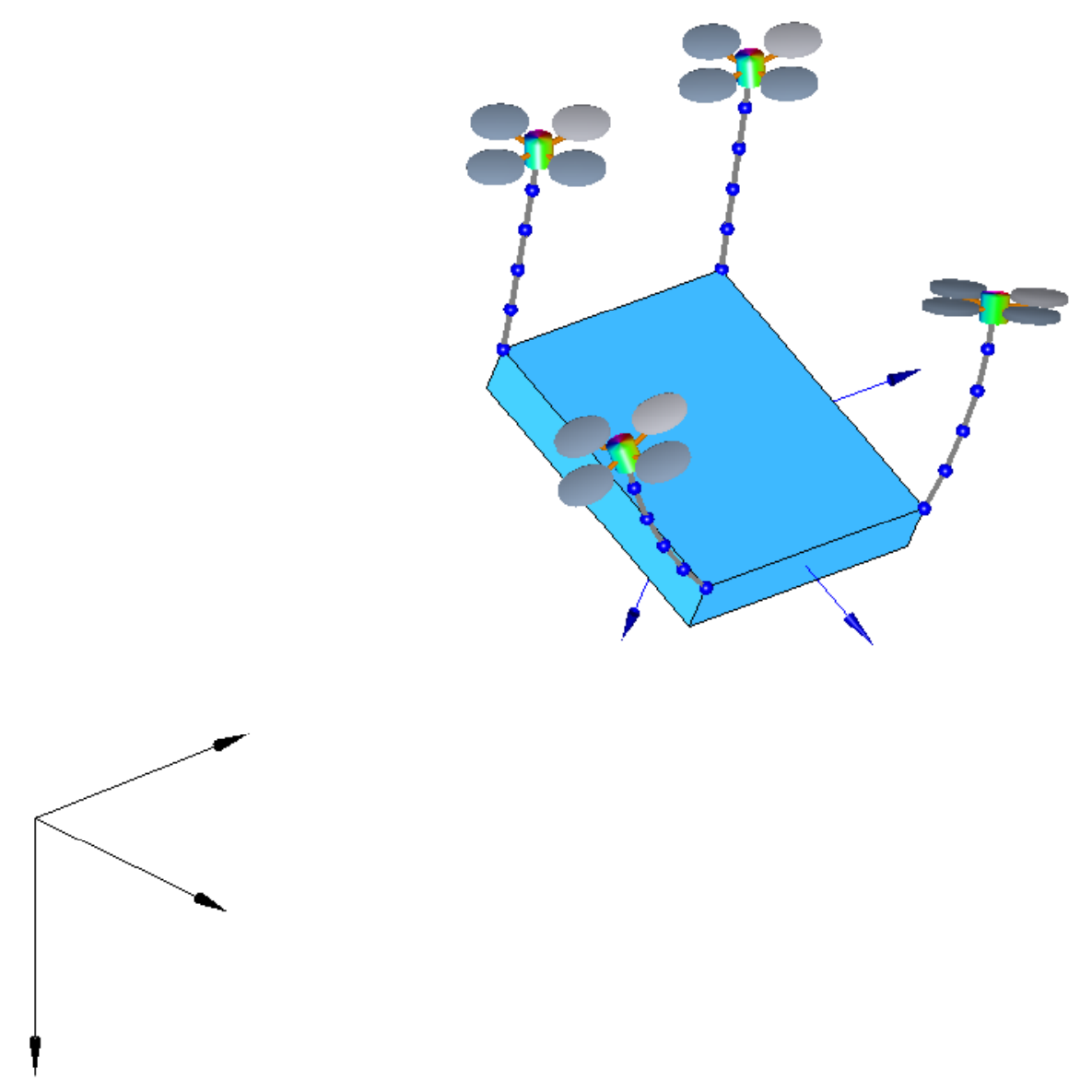}}
}
\centerline{
	\subfigure[$t=0.68$ Sec.]{
		\includegraphics[width=0.3\columnwidth]{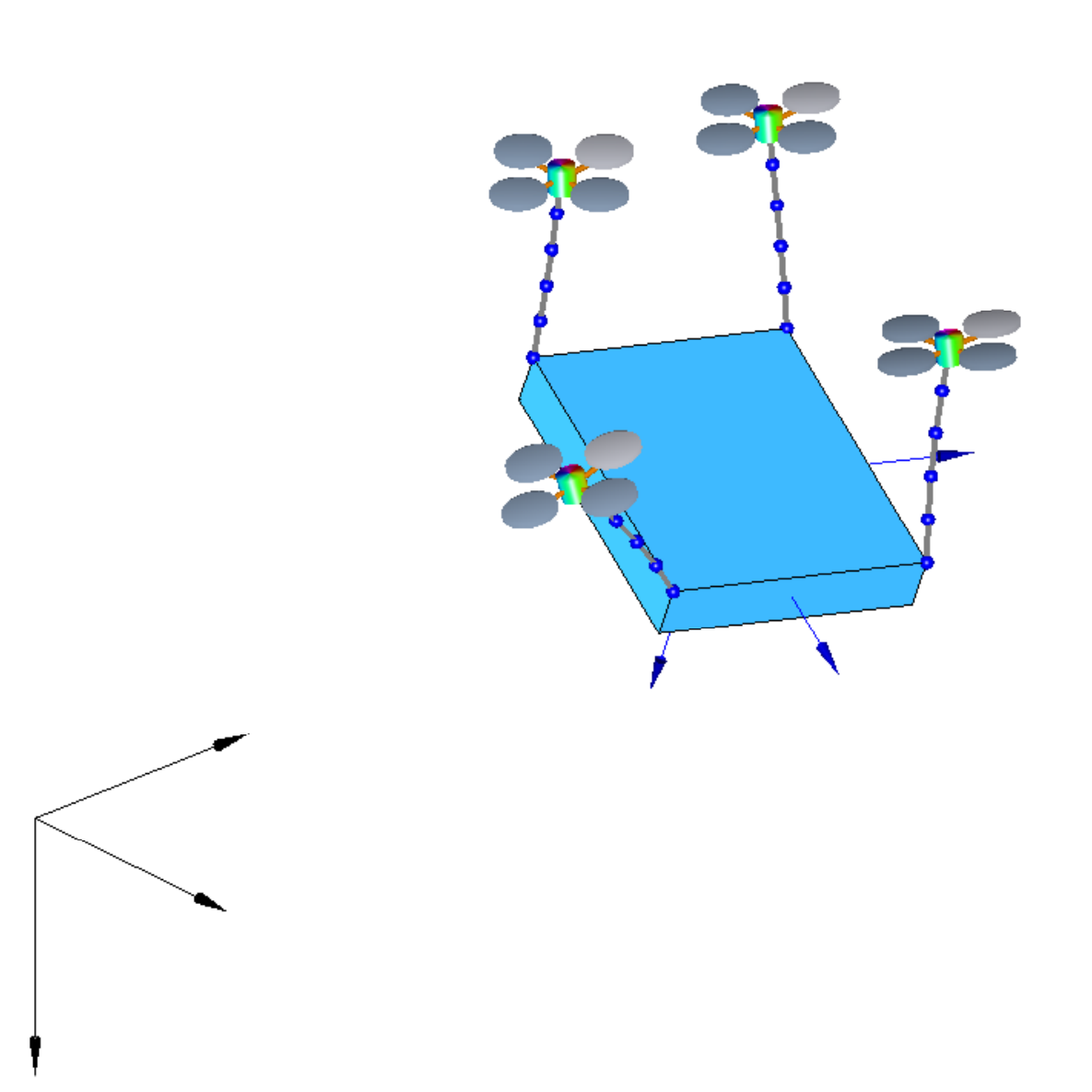}}
	\subfigure[$t=1.10$ Sec.]{
		\includegraphics[width=0.3\columnwidth]{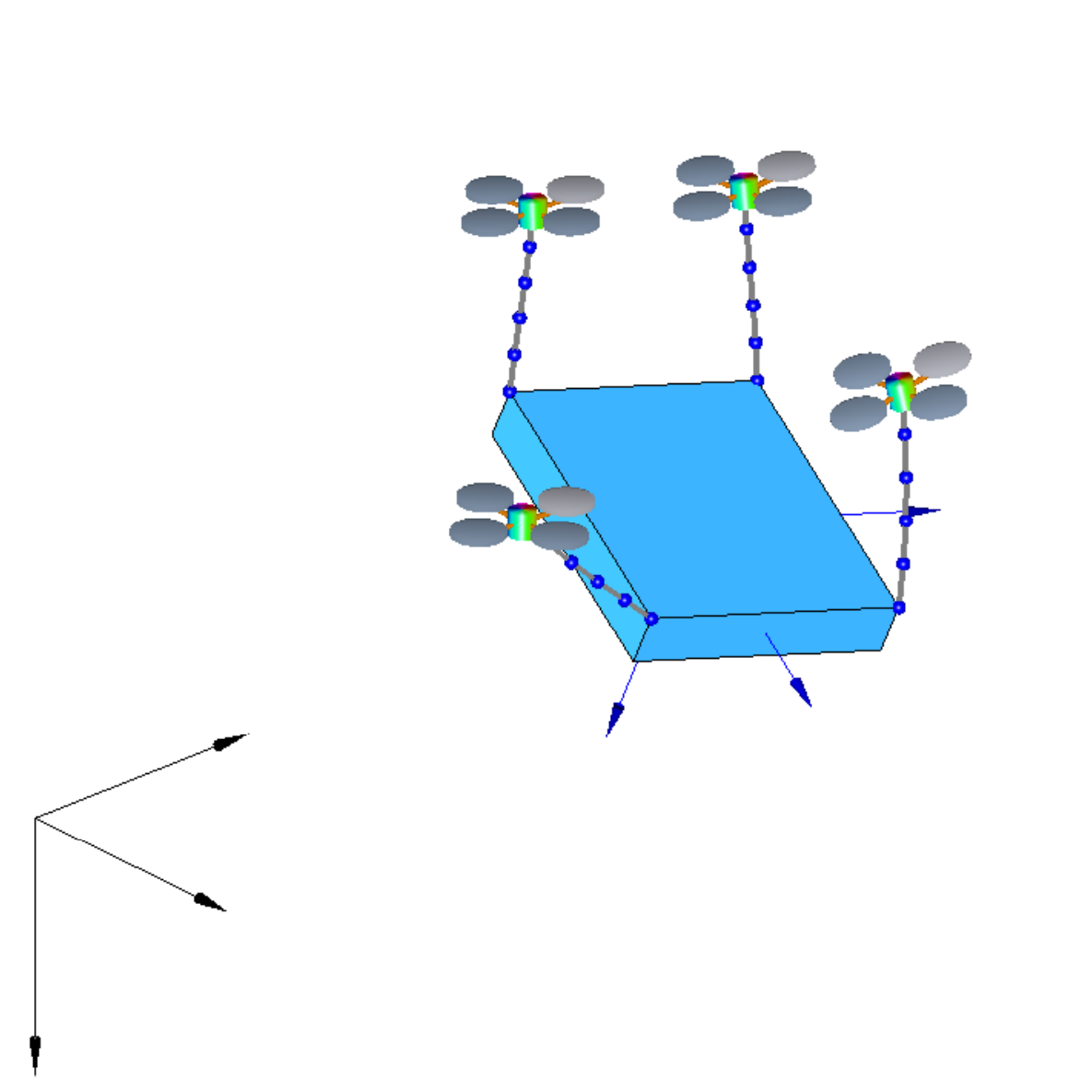}}
		\subfigure[$t=1.36$ Sec.]{
		\includegraphics[width=0.3\columnwidth]{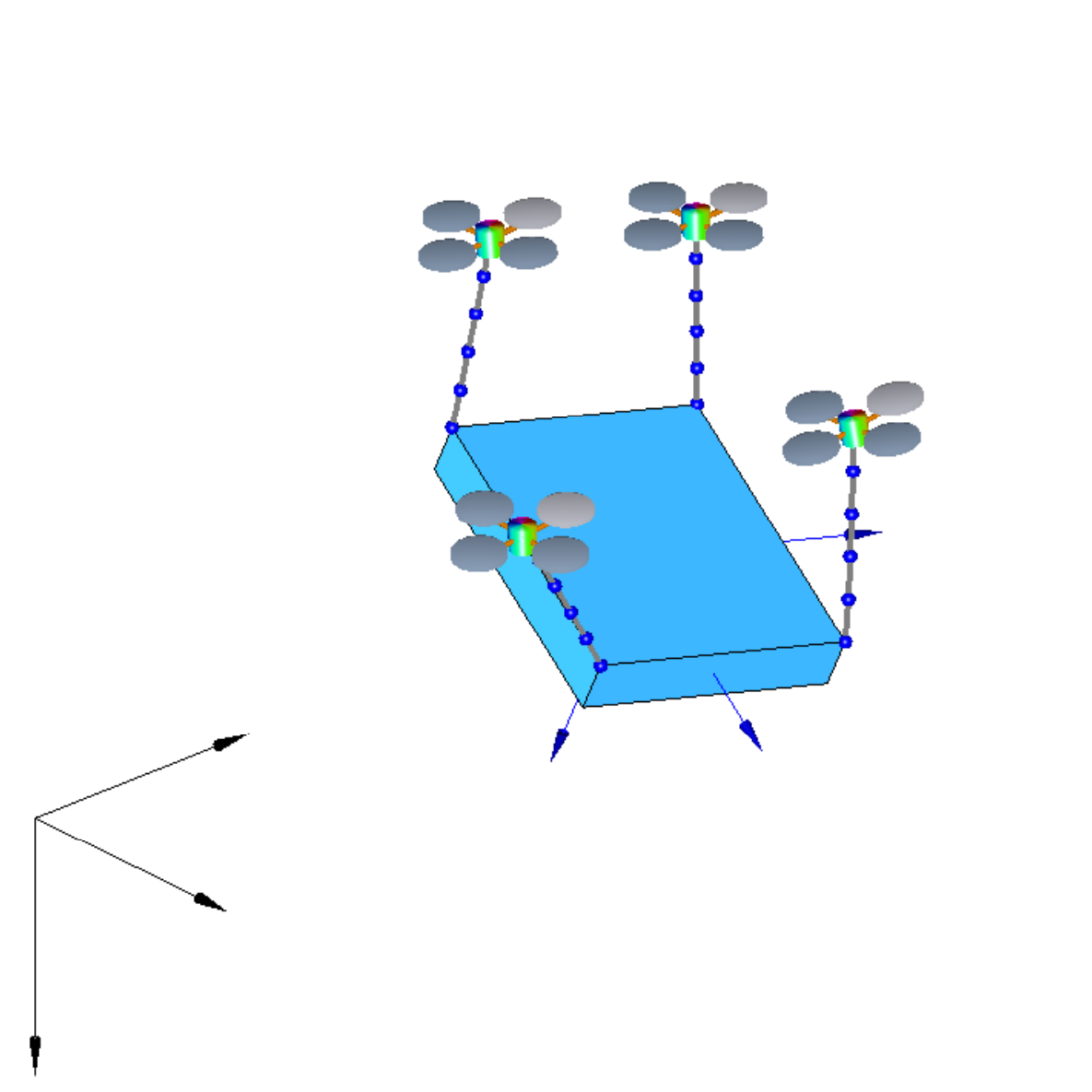}}
}
\centerline{
	\subfigure[$t=1.98$ Sec.]{
		\includegraphics[width=0.3\columnwidth]{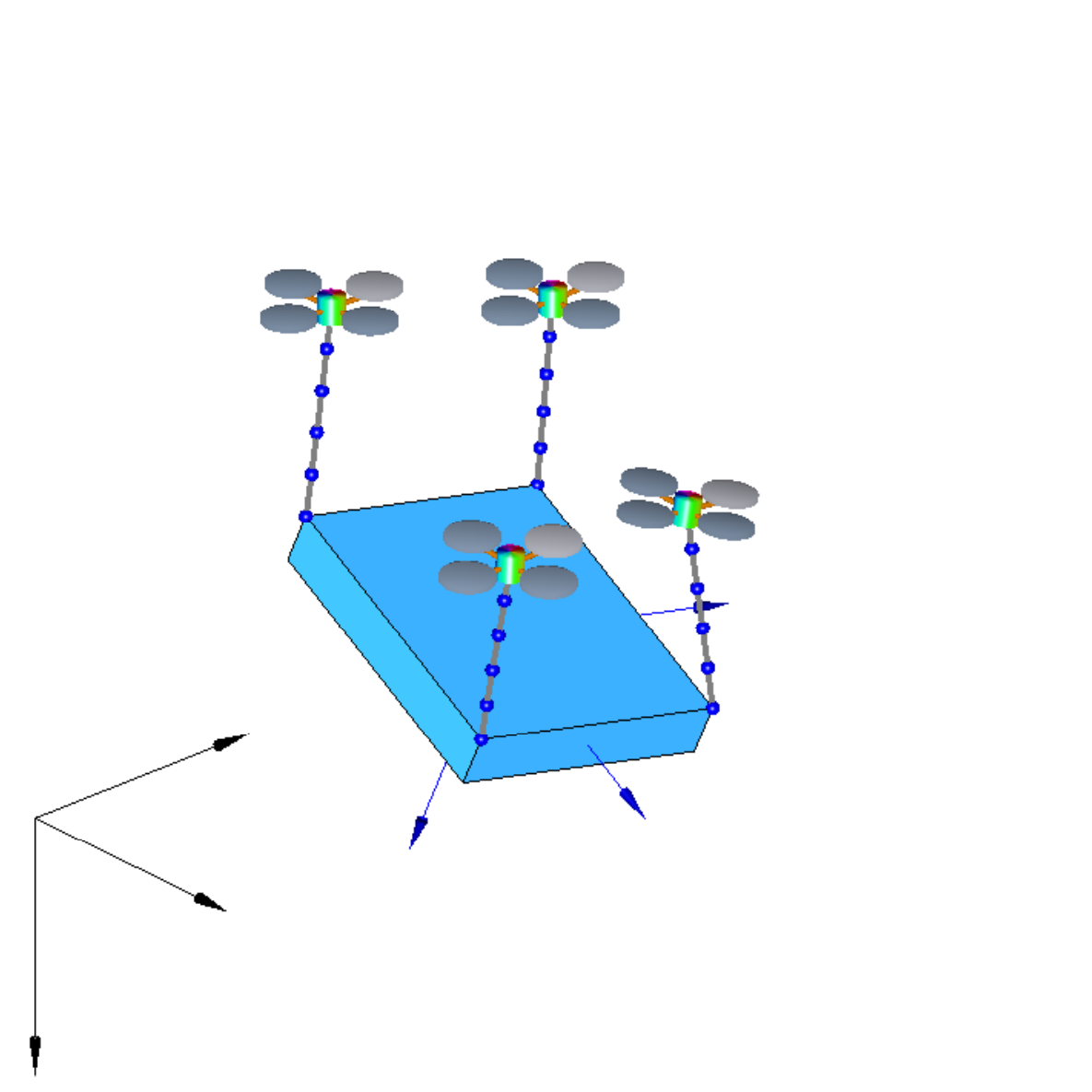}}
	\subfigure[$t=3.48$ Sec.]{
		\includegraphics[width=0.3\columnwidth]{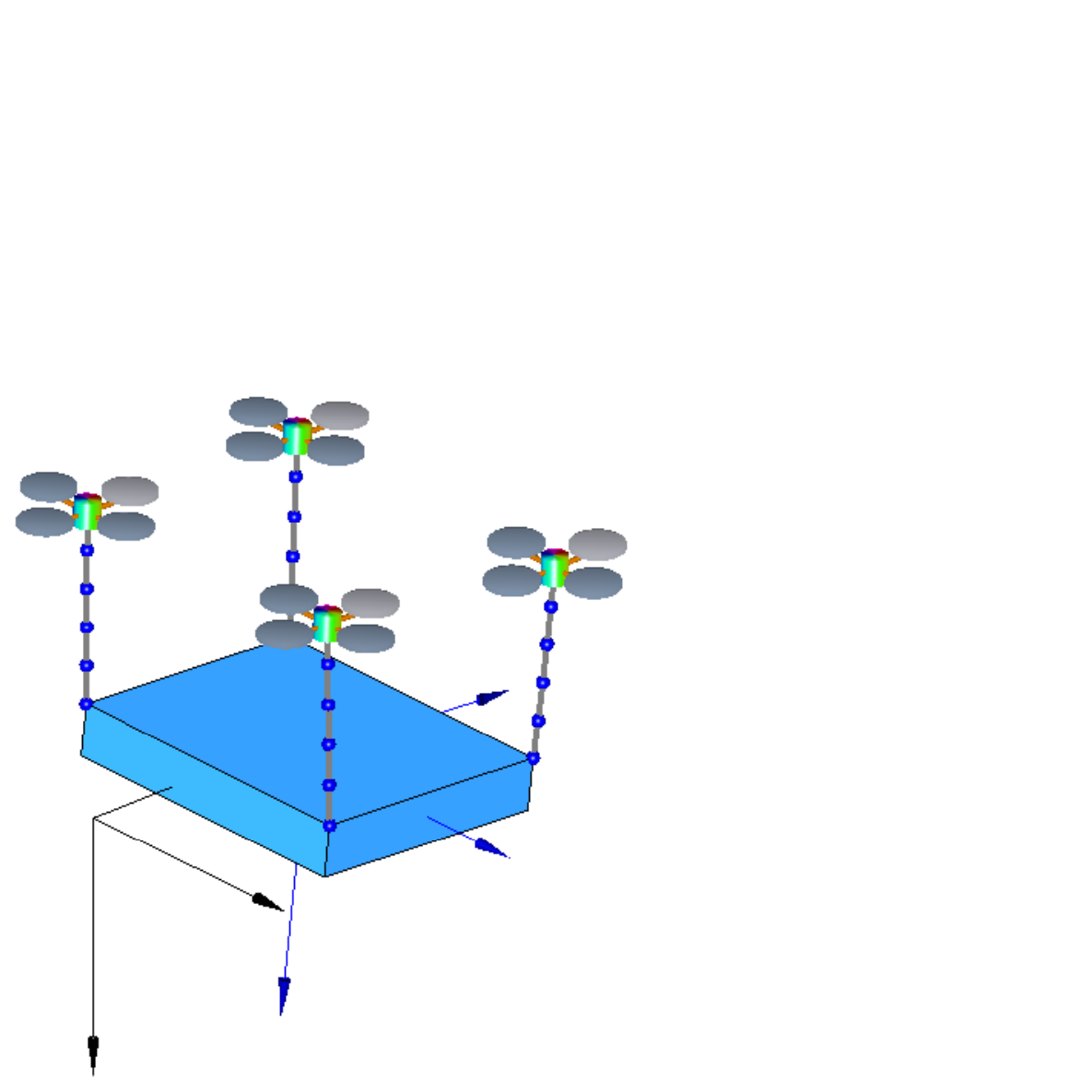}}
		\subfigure[$t=10$ Sec.]{
		\includegraphics[width=0.3\columnwidth]{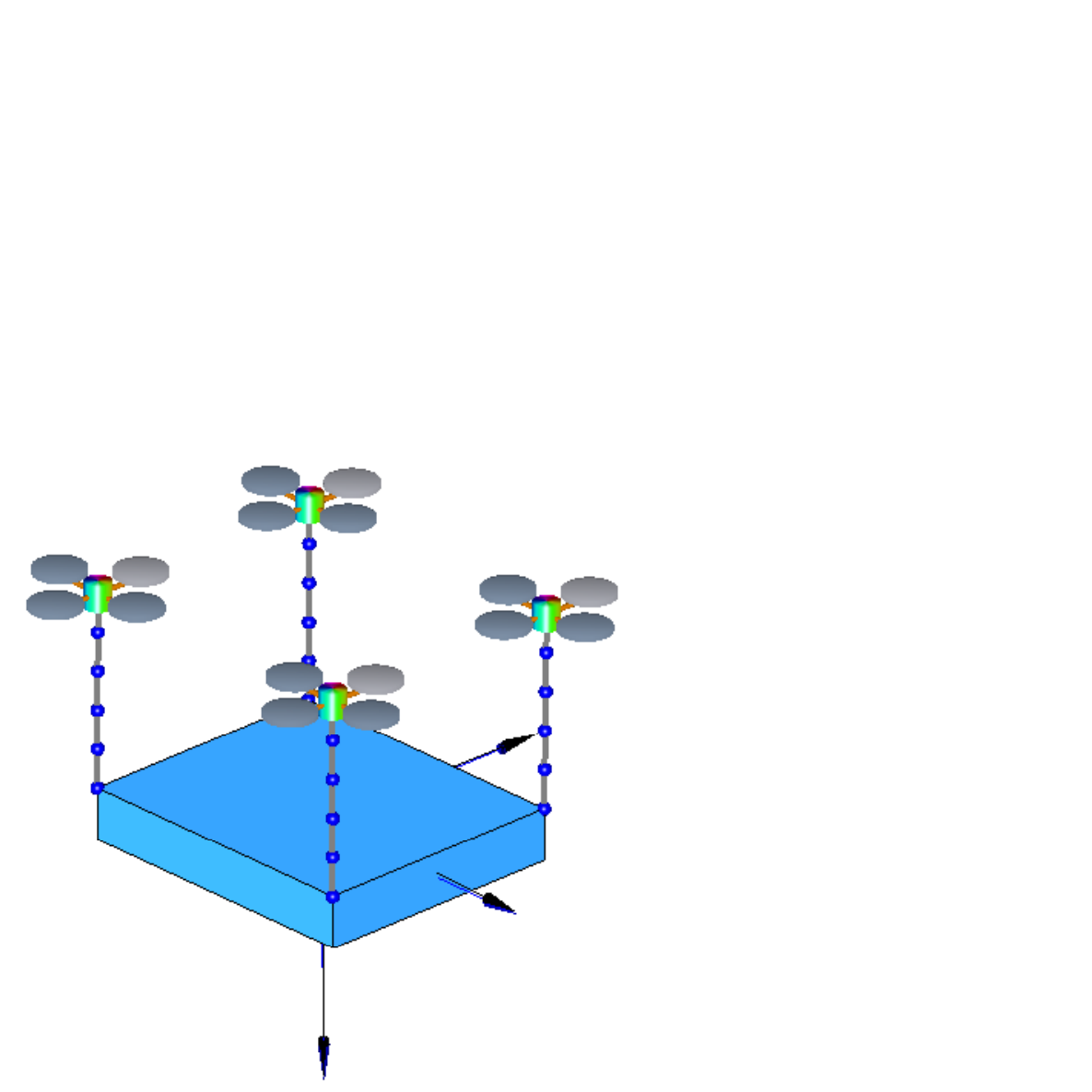}}
}
\caption{Snapshots of the controlled maneuver. A short animation is also available at  \href{http://youtu.be/Mp4Riw6xBl4}{http://youtu.be/Mp4Riw6xBl4}}
\label{fig:simresults3snap}
\end{figure}

Figure \ref{fig:simresults2} illustrates the tracking errors, and the total thrust of each quadrotor. Snapshots of the controlled maneuvers is also illustrated at Figure \ref{fig:simresults3snap}. It is shown that the proposed controller is able to stabilize the payload and cables at their desired configuration even from the large initial attitude errors.

%%%%%%%%%%%%%%%%%%%%%%%%%%%%%%%%%%%%%%%%%%%%%%%%%%%%%%%%%%%%%%%%
%\section{Conclusions}
%We utilized Euler-Lagrange equations to model multiple quadrotor UAVs, chain pendulums to model flexible cables, and rigid-body motion in 3D space. These derivations are developed in a remarkably compact form which allow us to choose an arbitrary number and any configuration of the links and an arbitrary number of quadrotors. We developed a geometric nonlinear controller to stabilize the links below the quadrotors and payload in the equilibrium position from an any chosen initial condition. We expanded these derivations in such a way that there is no need of using local angle coordinates which is an advantageous technique to signalize our derivations. A rigorous Lyapunov stability analysis is also presented to show that the zero equilibrium of the tracking errors is asymptotically stable. Numerical simulation illustrate the accuracy and extraordinary performance of the purposed model and control system.
%%%%%%%%%%%%%%%%%%%%%%%%%%%%%%%%%%%%%%%%%%%%%%%%%%%%%%%%%%%%%%%%

\appendix
\subsection{Proof for Proposition \ref{prop:FDM}}\label{sec:PfFDM}
\subsubsection{Kinetic Energy}
The kinetic energy of the whole system is composed of the kinetic energy of quadrotors, cables and the rigid body, as
\begin{align}
T=&\frac{1}{2}m_{0}\|\dot{x}_{0}\|^{2}+\sum_{i=1}^{n}\sum_{j=1}^{n_{i}}{\frac{1}{2}m_{ij}\|\dot{x}_{ij}\|^{2}}+\frac{1}{2}\sum_{i=1}^{n}{m_{i}\|\dot{x}_{i}\|^{2}}\nonumber\\
&+\frac{1}{2}\sum_{i=1}^{n}{\Omega_{i}\cdot J_{i}\Omega_{i}}+\frac{1}{2}\Omega_{0}\cdot J_{0}\Omega_{0}.
\end{align}
Substituting the derivatives of \refeqn{xi} and \refeqn{xij} into the above expression we have
\begin{align}
T=&\frac{1}{2}m_{0}\|\dot{x}_{0}\|^{2}+\sum_{i=1}^{n}\sum_{j=1}^{n_{i}}{\frac{1}{2}m_{ij}\|\dot{x}_{0}+\dot{R}_{0}\rho_{i}-\sum_{a=j+1}^{n_{i}}{l_{ia}\dot{q}_{ia}}\|^{2}} \nonumber \\
&+\frac{1}{2}\sum_{i=1}^{n}{m_{i}\|\dot{x}_{0}+\dot{R}_{0}\rho_{i}-\sum_{a=1}^{n_{i}}{l_{ia}\dot{q}_{ia}}\|^{2}} \nonumber\\
&+\frac{1}{2}\sum_{i=1}^{n}{\Omega_{i}\cdot J_{i}\Omega_{i}}+\frac{1}{2}\Omega_{0}\cdot J_{0}\Omega_{0}.
\end{align}
We expand the above expression as follow
\begin{align}\label{eqn:kineticbs}
T=&\frac{1}{2}(m_{0}\|\dot{x}_{0}\|^{2}+\sum_{i=1}^{n}\sum_{j=1}^{n_{i}}{m_{ij}\|\dot{x}_{0}\|^{2}}+\sum_{i=1}^{n}{m_{i}\|\dot{x}_{0}\|^2}) \nonumber\\
&+\frac{1}{2}\sum_{i=1}^{n}(\sum_{j=1}^{n_{i}}{m_{ij}\|\dot{R}_{0}\rho_{i}\|^2}+m_{i}\|\dot{R}_{0}\rho_{i}\|^2) \nonumber\\
&+\sum_{i=1}^{n}(\sum_{j=1}^{n_{i}}{m_{ij}\dot{x}_{0}\cdot\dot{R}_{0}\rho_{i}}+m_{i}\dot{x}_{0}\cdot\dot{R}_{0}\rho_{i}) \nonumber\\
&+\frac{1}{2}\sum_{i=1}^{n}(\sum_{j=1}^{n_{i}}{m_{ij}\|\sum_{a=j+1}^{n_{i}}{l_{ia}\dot{q}_{ia}}\|^2}+{m_{i}\|\sum_{a=1}^{n_{i}}{l_{ia}\dot{q}_{ia}}\|^2}) \nonumber\\
&-\sum_{i=1}^{n}(\sum_{j=1}^{n_{i}}{m_{ij}\dot{x}_{0}}\cdot\sum_{a=j+1}^{n_{i}}{{l_{ia}\dot{q}_{ia}}}+\dot{x}_{0}\cdot\sum_{a=1}^{n_{i}}{l_{ia}\dot{q}_{ia}}) \nonumber\\
&-\sum_{i=1}^{n}(\sum_{j=1}^{n_{i}}{m_{ij}\dot{R}_{0}\rho_{i}}\cdot\sum_{a=j+1}^{n_{i}}{l_{ia}\dot{q}_{ia}}+m_{i}\dot{R}_{0}\rho_{i}\cdot\sum_{a=1}^{n_{i}}{l_{ia}\dot{q}_{ia}}) \nonumber\\
&+\frac{1}{2}\sum_{i=1}^{n}{\Omega_{i}\cdot J_{i}\Omega_{i}}+\frac{1}{2}\Omega_{0}\cdot J_{0}\Omega_{0},
\end{align}
and substituting \refeqn{def1}, \refeqn{def3}, it is rewritten as
\begin{align}\label{eqn:kinetic}
T=&\frac{1}{2}M_{T}\|\dot{x}_{0}\|^2+\frac{1}{2}\sum_{i=1}^{n}{M_{iT}\|\dot{R}_{0}\rho_{i}\|^{2}}+\sum_{i=1}^{n}({M_{iT}\dot{x}_{0}\cdot\dot{R}_{0}\rho_{i}}) \nonumber\\
&+\sum_{i=1}^{n}\sum_{j,k=1}^{n_{i}}{M_{0ij}l_{ik}\dot{q}_{ij}\cdot\dot{q}_{ik}}-\sum_{i=1}^{n}({\dot{x}_{0}\cdot\sum_{j=1}^{n_{i}}{M_{0ij}l_{ij}\dot{q}_{ij}}}) \nonumber\\
&-\sum_{i=1}^{n}({\dot{R}_{0}\rho_{i}}\cdot\sum_{j=1}^{n_{i}}{M_{0ij}l_{ij}\dot{q}_{ij}}) \nonumber\\
&+\frac{1}{2}\sum_{i=1}^{n}{\Omega_{i}\cdot J_{i}\Omega_{i}}+\frac{1}{2}\Omega_{0}\cdot J_{0}\Omega_{0}.
\end{align}
%%%%%%%%%%%%%%%%%%%%%%%%%%%%%%%%%%%%%%%%%%%%%%%%%%%%%%%%%%%%%%%
%%%%%%%%%%%%%%%%%%%%%%%%%%%%%%%%%%%%%%%%%%%%%%%%%%%%%%%%%%%%%%%
\subsubsection{Potential Energy}
We can derive the potential energy expression by considering the gravitational forces on each part of system as given
\begin{align}
V=-m_{0}ge_{3}\cdot x_{0}-\sum_{i=1}^{n}{m_{i}ge_{3}}\cdot x_{i}-\sum_{i=1}^{n}\sum_{j=1}^{n_{i}}{m_{ij}ge_{3}}\cdot x_{ij}.
\end{align}
Using \refeqn{xi} and \refeqn{xij}, we obtain
\begin{align}
V=&-m_{0}ge_{3}\cdot x_{0}-\sum_{i=1}^{n}{m_{i}ge_{3}}\cdot (x_{0}+R_{0}\rho_{i}-\sum_{a=1}^{n_{i}}{l_{ia}q_{ia}}) \nonumber\\
&-\sum_{i=1}^{n}\sum_{j=1}^{n_{i}}{m_{ij}ge_{3}}\cdot (x_{0}+R_{0}\rho_{i}-\sum_{a=j+1}^{n_{i}}{l_{ia}q_{ia}}),
\end{align}
and utilizing \refeqn{def3}, we can simplify the potential energy as
\begin{align}
V=-M_{T}ge_{3}\cdot x_{0}-\sum_{i=1}^{n}{M_{iT}ge_{3}\cdot R_{0}\rho_{i}}+\sum_{i=1}^{n}\sum_{j=1}^{n_{i}}{M_{0ij}l_{ij}q_{ij}\cdot e_{3}}.
\end{align}
%%%%%%%%%%%%%%%%%%%%%%%%%%%%%%%%%%%%%%%%%%%%%%%%%%%%%%%%%%%%%%%
%%%%%%%%%%%%%%%%%%%%%%%%%%%%%%%%%%%%%%%%%%%%%%%%%%%%%%%%%%%%%%%
\subsubsection{Derivatives of Lagrangian}
We develop the equation of motion for the Lagrangian $L=T-V$. The derivatives of the Lagrangian are given by
\begin{align}
&D_{\dot{x}_{0}}L=M_{T}\dot{x}_{0}+\sum_{i=1}^{n}{M_{iT}\dot{R}_{0}\rho_{i}}-\sum_{i=1}^{n}\sum_{j=1}^{n_{i}}{M_{0ij}l_{ij}\dot{q}_{ij}},\\
&D_{x_{0}}L=M_{T}ge_{3},\\
&D_{\dot{q}_{ij}}L=\sum_{i=1}^{n}\sum_{j=1}^{n_{i}}{M_{0ij}l_{ik}\dot{q}_{ik}}-\sum_{i=1}^{n}{M_{0ij}l_{ij}(\dot{x}_{0}}+{\dot{R}_{0}\rho_{i}}),\\
&D_{q_{ij}}L=-\sum_{i=1}^{n}{M_{0ij}l_{ij}e_{3}},
\end{align}
where $D_{\dot x_0}$ denote the derivative with respect to $\dot x_0$, and other derivatives are defined similarly. We also have
\begin{align}
D_{\Omega_{0}}L=&J_{0}\Omega_{0}+\sum_{i=1}^{n}{M_{iT}\hat{\rho}_{i}R_{0}^{T}\dot{x}_{0}},\nonumber\\
&-\sum_{i=1}^{n}\sum_{j=1}^{n_{i}}{M_{0ij}l_{ij}\hat{\rho}_{i}R_{0}^{T}\dot{q}_{ij}}-\sum_{i=1}^{n}{M_{iT}\hat{\rho}_{i}^{2}\Omega_{0}},
\end{align}
\begin{align}
&D_{\Omega_{0}}L=\bar{J}_{0}\Omega_{0}+\sum_{i=1}^{n}{\hat{\rho}_{i}R_{0}^{T}(M_{iT}\dot{x}_{0}-\sum_{j=1}^{n_{i}}{M_{0ij}l_{ij}\dot{q}_{ij}})},\\
&D_{\Omega_{i}}L=\sum_{i=1}^{n}{J_{i}\Omega_{i}},
\end{align}
where $\bar{J}_{0}$ is defined as
\begin{align}
\bar{J}_{0}=J_{0}-\sum_{i=1}^{n}{M_{iT}\hat{\rho}_{i}^{2}}.
\end{align}
The derivation of the Lagrangian with respect to $R_{0}$ is given by
\begin{align}
D_{R_{0}}L\cdot\delta R_{0}=&\sum_{i=1}^{n}{M_{iT}R_{0}\hat{\eta}_{0}\hat{\Omega}_{0}\rho_{i}\cdot\dot{x}_{0}}\nonumber\\
&-\sum_{i=1}^{n}{R_{0}\hat{\eta}_{0}\hat{\Omega}_{0}\rho_{i}\cdot\sum_{j=1}^{n_{i}}{M_{0ij}l_{ij}\dot{q}_{ij}}}\nonumber\\
&+\sum_{i=1}^{n}{M_{iT}ge_{3}\cdot R_{0}\hat{\eta}_{0}\rho_{i}},
\end{align}
which can be rewritten as
\begin{align}
&D_{R_{0}}L\cdot\delta R_{0}=d_{R_{0}}\cdot \eta_{0},
\end{align}
where
\begin{align}
d_{R_{0}}=\sum_{i=1}^{n}&(((\widehat{\hat{\Omega}_{0}\rho_{i}}R_{0}^{T}(M_{iT}\dot{x}_{0})-\sum_{j=1}^{n_{i}}{M_{0ij}l_{ij}\dot{q}_{ij}})\nonumber\\
&+M_{iT}g\hat{\rho}_{i}R_{0}^{T}e_{3})).
\end{align}
\subsubsection{Lagrange-d'Alembert Principle}
Consider $\mathfrak{G}=\int_{t_{0}}^{t_{f}}{L}$ be the action integral. Using the equations derived in previous section, the infinitesimal variation of the action integral can be written as
\begin{align}
\delta \mathfrak{G}=&\int_{t_{0}}^{t_{f}}D_{\dot{x}_{0}}L\cdot\delta\dot{x}_{0}+D_{x_{0}}\cdot\delta x_{0} \nonumber\\
&+D_{\Omega_{0}}L(\dot{\eta}_{0}+\Omega_{0}\times \eta_{0})+d_{R_{0}}L\cdot\eta_{0} \nonumber\\
&+\sum_{i=1}^{n}\sum_{j=1}^{n_{i}}{D_{\dot{q}_{ij}}L(\dot{\xi}_{ij}\times q_{ij}+\xi_{ij}\times\dot{q}_{ij})}\nonumber\\
&+\sum_{i=1}^{n}\sum_{j=1}^{n_{i}}{D_{q_{ij}}L\cdot(\xi_{ij}\times q_{ij})} \nonumber\\
&+\sum_{i=1}^{n}{D_{\Omega_{i}}L\cdot(\dot{\eta}_{i}+\Omega_{i}\times\eta_{i})}.
\end{align}
The total thrust at the $i$-th quadrotor with respect to the inertial frame is denoted by $u_{i}=-f_{i}R_{i}e_{3}\in\Re^{3}$ and the total moment at the $i$-th quadrotor is defined as $M_{i}\in\Re^{3}$. The corresponding virtual work is given by
\begin{align}
\delta W=\int_{t_{0}}^{t_{f}}&{\sum_{i=1}^{n}{u_{i}\cdot\{\delta x_{0}+R_{0}\hat{\eta}_{0}\rho_{i}-\sum_{j=1}^{n_{i}}{l_{ij}\dot{\xi}_{ij}\times q_{ij}}}}\} \nonumber\\
&+M_{i}\cdot \eta_{i}\; dt.
\end{align}
According to Lagrange-d Alembert principle, we have $\delta \mathfrak{G}=-\delta W$ for any variation of trajectories with fixes end points. By using integration by parts and rearranging, we obtain the following Euler-Lagrange equations
\begin{gather}
\frac{d}{dt}D_{\dot{x}_{i}}L-D_{x_{0}}L=\sum_{i=1}^{n}{u_{i}},\\
\frac{d}{dt}D_{\Omega_{0}}+\Omega_{0}\times D_{\Omega_{0}}-d_{R_{0}}=\sum_{i=1}^{n}{\hat{\rho}_{i}R_{0}^{T}u_{i}},\\
\hat{q}_{ij}\frac{d}{dt}D_{\dot{q}_{ij}}L-\hat{q}_{ij}D_{q_{i}}L=-l_{ij}\hat{q}_{ij}u_{i},\\
\frac{d}{dt}D_{\Omega_{i}}L+\Omega_{i}\times D_{\Omega_{i}}L=M_{i}.
\end{gather}
Substituting the derivatives of Lagrangians into the above expression and rearranging, the equations of motion are given by \refeqn{EOMM1}, \refeqn{EOMM2}, \refeqn{EOMM3}, \refeqn{EOMM4}.

\subsection{Proof for Proposition \ref{prop:stability1}}\label{sec:P1stability}
The variations of $x$ and $q$ are given by \refeqn{xlin} and \refeqn{qlin}. From the kinematics equation $\dot q_{ij}=\omega_{ij}\times q_{ij}$ and
\begin{align*}
\delta \dot q_{ij} = \dot\xi_{ij} \times e_3 =\delta\omega_{ij} \times e_3 + 0\times (\xi_{ij}\times e_3)=\delta\omega_{ij} \times e_3.
\end{align*}
Since both sides of the above equation is perpendicular to $e_3$, this is equivalent to $e_3\times(\dot\xi_{ij}\times e_3) = e_3\times(\delta\omega_{ij}\times e_3)$, which yields
\begin{gather*}
\dot \xi_{ij} - (e_3\cdot\dot\xi_{ij}) e_3 = \delta\omega_{ij} -(e_3\cdot\delta\omega_{ij})e_3.
\end{gather*}
Since $\xi_{ij}\cdot e_3 =0$, we have $\dot\xi_{ij}\cdot e_3=0$. As $e_3\cdot\delta\omega_{ij}=0$ from the constraint, we obtain the linearized equation for the kinematics equation of the link
\begin{align}
\dot\xi_{ij} = \delta\omega_{ij}.\label{eqn:dotxii}
\end{align}
The infinitesimal variation of $R_{0}\in\SO$ in terms of the exponential map
\begin{align}
\delta R_{0} = \frac{d}{d\epsilon}\bigg|_{\epsilon = 0} R_{0}\exp (\epsilon \hat\eta_{0}) = R_{0}\hat\eta_{0},\label{eqn:delR0}
\end{align}
for $\eta_{0}\in\Re^3$. Substituting these into \refeqn{EOMM1}, \refeqn{EOMM2}, and \refeqn{EOMM3}, and ignoring the higher order terms, we obtain the following sets of linearized equations of motion 
\begin{gather}
M_{T}\delta \ddot{x}_{0}-\sum_{i=1}^{n}{M_{iT}\hat{\rho}_{i}}\delta\dot{\Omega}_{0}\nonumber\\
+\sum_{i=1}^{n}\sum_{j=1}^{n_{i}}M_{0ij}l_{ij}\hat{e}_{3}C(C^{T}\ddot{\xi}_{ij})=\sum_{i=1}^{n}{\delta u_{i}}\\
\sum_{i=1}^{n}{M_{iT}\hat{\rho}_{i}\delta\ddot{x}_{0}}+\bar{J}_{0}\delta\dot{\Omega}_{0}+\sum_{i=1}^{n}\sum_{j=1}^{n_{i}}M_{0ij}l_{ij}\hat{\rho}_{i}\hat{e}_{3}C(C^{T}\ddot{\xi}_{ij})\nonumber\\
+\sum_{i=1}^{n}\frac{m_{0}}{n}g\hat{\rho}_{i}\hat{e}_{3}\eta_{0}=\sum_{i=1}^{n}{\hat{\rho}_{i}\delta u_{i}}\\
-M_{0ij}C^{T}\hat{e}_{3}\delta\ddot{x}_{0}+M_{0ij}C^{T}\hat{e}_{3}\hat{\rho}_{i}\delta\dot{\Omega}_{0}+\sum_{k=1}^{n_{i}}{M_{0ij}l_{ik}I_{2}(C^{T}\ddot{\xi}_{ij})}\nonumber\\
=-C^{T}\hat{e}_{3}\delta u_{i}+(-M_{iT}-\frac{m_{0}}{n}+M_{0ij})ge_{3} I_{2}(C^{T}\xi_{ij})\\
\dot{\eta}_{i}=\delta\Omega_{i},\; \dot{\eta}_{0}=\delta\Omega_{0},\; J_{i}\delta\Omega_{i}=\delta M_{i},
\end{gather}
which can be written in a matrix form as presented in \refeqn{EOMLin}. See~\cite{FarhadDTLeeIJCAS} for detaied derivations for a similar dynamic system. We used $C^{T}\hat{e}_{3}^{2}C=-I_{2}$ to simplify these derivations. 

\subsection{Proof for Proposition \ref{prop:stability}}\label{sec:Pstabilityddd}
We first show stability of the rotational dynamics of each quadrotor, and later it is combined with the stability analysis for the remaining parts.
\subsubsection{Attitude Error Dynamics}
Here, attitude error dynamics for $e_{R_{i}}$, $e_{\Omega_{i}}$ are derived and we find conditions on control parameters to guarantee the boundedness of attitude tracking errors. The time-derivative of $J_{i}e_{\Omega_{i}}$ can be written as
\begin{align}
J_{i}\dot e_{\Omega_{i}} & = \{J_{i}e_{\Omega_{i}} + d_{i}\}^\wedge e_{\Omega_{i}} - k_R e_{R_{i}}-k_\Omega e_{\Omega_{i}},\label{eqn:JeWdot}
\end{align}
where $d_{i}=(2J_{i}-\trs{[J_{i}]I})R_{i}^TR_{i_{d}}\Omega_{i_d}\in\Re^3$~\cite{Farhad2013}. The important property is that the first term of the right hand side is normal to $e_{\Omega_{i}}$, and it simplifies the subsequent Lyapunov analysis.

\subsubsection{Stability for Attitude Dynamics}
Define a configuration error function on $\SO$ as follows
\begin{align}
\Psi_{i}= \frac{1}{2}\trs[I- R_{{i}_c}^T R_{i}].
\end{align}
We introduce the following Lyapunov function
\begin{align}
\mathcal{V}_{2}=\sum_{i=1}^{n}{\mathcal{V}_{2_{i}}},
\end{align}
where 
\begin{align}
\mathcal{V}_{2_i}=\frac{1}{2}e_{\Omega_{i}}\cdot J_{i}\dot{e}_{\Omega_{i}}+k_{R}\Psi_{i}(R_{i},R_{d_{i}})+c_{2_i}e_{R_{i}}\cdot e_{\Omega_{i}}.
\end{align}
Consider a domain $D_{2}$ given by
\begin{align}
D_2 = \{ (R_{i},\Omega_{i})\in \SO\times\Re^3\,|\, \Psi_{i}(R_{i},R_{d_{i}})<\psi_{2_i}<2\}.\label{eqn:D2}
\end{align}
In this domain we can show that $\mathcal{V}_{2}$ is bounded as follows~\cite{Farhad2013}
\begin{align}\label{eqn:ffff1}
\begin{split}
z_{2_i}^{T}M_{i_{21}}z_{2_i}\leq\mathcal{V}_{2_i}\leq z_{2_i}^{T}M_{i_{22}}z_{2_i},
\end{split}
\end{align}
and $z_{2_i}=[\|e_{R_{i}}\|,\|e_{\Omega_{i}}\|]^{T}\in \Re^{2}$. Matrices $M_{i_{21}}$, $M_{i_{22}}$ are given by
\begin{align}
M_{i_{21}}=&\frac{1}{2}\begin{bmatrix}
k_{R}&-c_{2_i}\lambda_{M_{i}}\\
-c_{2_i}\lambda_{M_{i}}&\lambda_{m_{i}}
\end{bmatrix},\nonumber\\
M_{i_{22}}=&\frac{1}{2}\begin{bmatrix}
\frac{2k_{R}}{2-\psi_{2_i}}&c_{2_i}\lambda_{M_{i}}\\
c_{2_i}\lambda_{M_{i}}&\lambda_{M_{i}}\nonumber
\end{bmatrix},
\end{align}
The time derivative of $\mathcal{V}_2$ along the solution of the controlled system is given by
\begin{align*}
\dot{\mathcal{V}}_2  =&
\sum_{i=1}^{n}-k_\Omega\|e_{\Omega_{i}}\|^2 + c_{2_i} \dot e_{R_{i}} \cdot J_{i}e_{\Omega_{i}}+ c_{2_i} e_{R_{i}} \cdot J_{i}\dot e_{\Omega_{i}}.
\end{align*}
Substituting \refeqn{JeWdot}, the above equation becomes
\begin{align*}
\dot{\mathcal{V}}_2 =&
\sum_{i=1}^{n}-k_\Omega\|e_{\Omega_{i}}\|^2  + c_{2_i} \dot e_{R_{i}} \cdot J_{i}e_{\Omega_{i}}-c_{2_i} k_R \|e_{R_{i}}\|^2 \\
&+ c_{2_i} e_{R_{i}} \cdot ((J_{i}e_{\Omega_{i}}+d_{i})^\wedge e_{\Omega_{i}} -k_\Omega e_{\Omega_{i}}).
\end{align*}
We have $\|e_{R_{i}}\|\leq 1$, $\|\dot e_{R_{i}}\|\leq \|e_{\Omega_{i}}\|$~\cite{TFJCHTLeeHG}, and choose a constant $B_{2_i}$ such that $\|d_{i}\|\leq B_{i_2}$. Then we have
\begin{align}
\dot{\mathcal{V}}_2 \leq -\sum_{i=1}^{n} z_{2_i}^T W_{2_i} z_{2_i},\label{eqn:dotV2}
\end{align}
where the matrix $W_{2_i}\in\Re^{2\times 2}$ is given by
\begin{align*}
W_{2_i} = \begin{bmatrix} c_{2_i}k_R & -\frac{c_{2_i}}{2}(k_\Omega+B_{2_i}) \\ 
-\frac{c_{2_i}}{2}(k_\Omega+B_{2_i}) & k_\Omega-2c_{2_i}\lambda_{M_{i}} \end{bmatrix}.
\end{align*}
The matrix $W_{2_i}$ is a positive definite matrix if 
\begin{align}\label{eqn:c2}
c_{2_i}<\min\{\frac{\sqrt{k_{R}\lambda_{m_i}}}{\lambda_{M_i}},\frac{4k_{\Omega}}{8k_{R}\lambda_{M_i}+(k_{\Omega}+B_{i_2})^{2}} \}.
\end{align}
This implies that
\begin{align}\label{eqn:eq2}
\dot{\mathcal{V}}_{2}\leq -\sum_{i=1}^{n} {\lambda_{m}(W_{2_i})\|z_{2_i}\|^{2}},
\end{align} 
which shows stability of the attitude dynamics of quadrotors.
%%%%%%%%%%%%%%%%%%%%%%%%%%%%%%%%%%%%%%%%
\subsubsection{Error Dynamics of the Payload and Links}
We derive the tracking error dynamics and a Lyapunov function for the translational dynamics of a payload and the dynamics of links. Later it is combined with the stability analyses of the rotational dynamics. From \refeqn{EOMM1}, \refeqn{EOMLin}, \refeqn{Ai}, and \refeqn{fi}, the equation of motion for the controlled dynamic model is given by
\begin{align}\label{eqn:salam}
\Mb\ddot \xb  + \Gb\xb=\Bb(u-u^{*})+\g(\xb,\dot{\xb}),
\end{align}
where
\begin{align}
u=\begin{bmatrix}
u_{1}\\
u_{2}\\
\vdots\\
u_{n}
\end{bmatrix},\; u^{*}=\begin{bmatrix}
-(M_{1T}+\frac{m_{0}}{n})ge_{3}\\
-(M_{2T}+\frac{m_{0}}{n})ge_{3}\\
\vdots\\
-(M_{nT}+\frac{m_{0}}{n})ge_{3}
\end{bmatrix},
\end{align}
and $\g(\xb,\dot{\xb})$ corresponds to the higher order terms. As $u_i=-f_iR_ie_3$ for the full dynamic model, $\delta u=u-u^*$ is given by
\begin{align}\label{eqn:deltauu}
\delta u=
\begin{bmatrix}
-f_{1}R_{1}e_{3}+(M_{1T}+\frac{m_{0}}{n})ge_{3}\\
-f_{2}R_{2}e_{3}+(M_{2T}+\frac{m_{0}}{n})ge_{3}\\
\vdots\\
-f_{n}R_{n}e_{3}+(M_{nT}+\frac{m_{0}}{n})ge_{3}
\end{bmatrix}.
\end{align}
The subsequent analyses are developed in the domain $D_{1}$
\begin{align}
D_1=\{&(\xb,\dot{\xb},R_i,e_{\Omega_i})\in\Re^{D_{\xb}}\times\Re^{D_{\xb}}\times \SO\times\Re^3\,|\,\nonumber\\
& \Psi_{i}< \psi_{1_i} < 1\}.\label{eqn:D}
\end{align}
In the domain $D_{1}$, we can show that 
\begin{align}
\frac{1}{2} \norm{e_{R_{i}}}^2 \leq  \Psi_i(R_i,R_{c_i}) \leq \frac{1}{2-\psi_{1_i}} \norm{e_{R_i}}^2\label{eqn:eRPsi1}.
\end{align}
Consider the quantity $e_{3}^{T}R_{c_i}^{T}R_{i}e_{3}$, which represents the cosine of the angle between $b_{3_i}=R_{i}e_{3}$ and $b_{3_{c_i}}=R_{c_i}e_{3}$. Since $1-\Psi_i(R_i,R_{c_i})$ represents the cosine of the eigen-axis rotation angle between $R_{c_i}$ and $R_i$, we have $e_{3}^{T}R_{c_i}^{T}Re_{3}\geq 1-\Psi_i(R_i,R_{c_i})>0$ in $D_{1}$. Therefore, the quantity $\frac{1}{e_{3}^{T}R_{c_i}^{T}R_i e_{3}}$ is well-defined. We add and subtract $\frac{f_i}{e_{3}^{T}R_{c_i}^{T}R_i e_{3}}R_{c_i}e_{3}$ to the right hand side of \refeqn{deltauu} to obtain
\begin{align}\label{eqn:hallaa}
\delta u=
\begin{bmatrix}
\frac{-f_1}{e_{3}^{T}R_{c_1}^{T}R_1 e_{3}}R_{c_1}e_{3}-X_1+(M_{1T}+\frac{m_{0}}{n})ge_{3}\\
\frac{-f_2}{e_{3}^{T}R_{c_2}^{T}R_2 e_{3}}R_{c_2}e_{3}-X_2+(M_{2T}+\frac{m_{0}}{n})ge_{3}\\
\vdots\\
\frac{-f_n}{e_{3}^{T}R_{c_n}^{T}R_n e_{3}}R_{c_n}e_{3}-X_n+(M_{nT}+\frac{m_{0}}{n})ge_{3}
\end{bmatrix}.
\end{align}
%\begin{align}\label{eqn:taghall}
%\Mb\ddot \xb  + \Gb\xb=\Bb(&\sum_{i=1}^{n}(\frac{-f_i}{e_{3}^{T}R_{c_i}^{T}R_i e_{3}}R_{c_i}e_{3}-X_i)\nonumber\\
%&-M_{T}ge_{3})+\g(\xb,\dot{\xb}),
%\end{align}
where $X_i \in \Re^{3}$ is defined by
\begin{align}\label{eqn:Xdef}
X_i=\frac{f_i}{e_{3}^{T}R_{c_i}^{T}R_i e_{3}}((e_{3}^{T}R_{c_i}^{T}R_i e_{3})R_i e_{3}-R_{c_i}e_{3}).
\end{align}
Using 
\begin{align}
-\frac{f_i}{e_{3}^{T}R_{c_i}^{T}R_i e_{3}}R_{c_i}e_{3}=-\frac{(\|A_i\|R_{c_i}e_{3})\cdot R_i e_{3}}{e_{3}^{T}R_{c_i}^{T}R_i e_{3}}\cdot -\frac{A_i}{\|A_i\|}=A_i,
\end{align}
the equation \refeqn{hallaa} becomes
\begin{align}
\delta u=
\begin{bmatrix}
A_{1}-X_1+(M_{1T}+\frac{m_{0}}{n})ge_{3}\\
A_{2}-X_2+(M_{2T}+\frac{m_{0}}{n})ge_{3}\\
\vdots\\
A_{n}-X_n+(M_{nT}+\frac{m_{0}}{n})ge_{3}
\end{bmatrix}.
\end{align}
Substituting \refeqn{Ai} into the above equation, \refeqn{salam} becomes
\begin{align}
\Mb\ddot \xb  + \Gb\xb=\Bb(-K_{\xb}\xb-K_{\dot{\xb}}\dot{\xb}-X)+\g(\xb,\dot{\xb}),
\end{align}
where $X=[X_{1}^T,\; X_{2}^T,\; \cdots,\; X_{n}^T]^{T}\in\Re^{3n}$.
It is rewritten in the following matrix form
\begin{align}\label{eqn:zdot1}
\dot{z}_{1}=\mathds{A} z_{1}+\mathds{B}(\Bb X+\g(\xb,\dot{\xb})),
\end{align}
where $z_{1}=[\xb,\dot{\xb}]^{T}\in\Re^{2D_{\xb}}$ and
\begin{align}
\mathds{A}=\begin{bmatrix}
0&I\\
-\Mb^{-1}(\Gb+\Bb K_{\xb})&-\Mb^{-1}\Bb K_{\dot{\xb}}
\end{bmatrix},
\mathds{B}=\begin{bmatrix}
0\\
\Mb^{-1}
\end{bmatrix}.
\end{align}
We can also choose $K_{\xb}$ and $K_{\dot{\xb}}$ such that $\mathds{A}\in\Re^{2D_{x}\times 2D_{x}}$ is Hurwitz. Then for any positive definite matrix $Q\in\Re^{2D_{\xb}\times 2D_{\xb}}$, there exist a positive definite and symmetric matrix $P\in\Re^{2D_{\xb}\times 2D_{\xb}}$ such that $\mathds{A}^{T}P+P\mathds{A}=-Q$ according to~\cite[Thm 3.6]{Kha96}.
\subsubsection{Lyapunov Candidate for Simplified Dynamics}
From the linearized control system developed at section 3, we use matrix $P$ to introduce the following Lyapunov candidate for translational dynamics
\begin{align}
\mathcal{V}_{1}=z_{1}^{T}Pz_{1}.
\end{align}
The time derivative of the Lyapunov function using the Leibniz integral rule is given by
\begin{align}\label{eqn:devrr}
\dot{\mathcal{V}_{1}}=\dot{z}_{1}^{T}Pz_{1}+z_{1}^{T}P\dot{z}_{1}.
\end{align}
Substituting \refeqn{zdot1} into above expression
\begin{align}\label{eqn:beforsimp}
\dot{\mathcal{V}}_{1}=z_{1}^{T}(\mathds{A}^{T}P+P\mathds{A})z_{1}+2z_{1}^{T}P\mathds{B}(\Bb X+\g(\xb,\dot{\xb})).
\end{align}
Let $c_{3}=2\|P\mathds{B}\Bb\|_{2}\in\Re$ and using $\mathds{A}^{T}P+P\mathds{A}=-Q$, we have
\begin{align}\label{eqn:test}
\dot{\mathcal{V}}_{1}\leq-z_{1}^{T}Qz_{1}+c_{3}\|z_{1}\|\|X\|+2z_{1}^{T}P\mathds{B}\g(\xb,\dot{\xb}).
\end{align}
The second term on the right hand side of the above equation corresponds to the effects of the attitude tracking error on the translational dynamics. We find a bound of $X_{i}$, defined at \refeqn{Xdef}, to show stability of the coupled translational dynamics and rotational dynamics in the subsequent Lyapunov analysis. Since 
\begin{align}
f_{i}=\|A_{i}\|(e_{3}^{T}R_{c_i}^{T}R_i e_{3}), 
\end{align}
we have
\begin{align}\label{eqn:ssstr}
\|X_i\|\leq\|A_i\|\|(e_{3}^{T}R_{c_i}^{T}R_i e_{3})R_i e_{3}-R_{c_i}e_{3}\|.
\end{align}
The last term $\|(e_{3}^{T}R_{c_i}^{T}R_i e_{3})R_i e_{3}-R_{c_i}e_{3}\|$ represents the sine of the angle between $b_{3_i}=R_i e_{3}$ and $b_{3_{c_i}}=R_{c_i}e_{3}$, since $(b_{3_{c_i}}\cdot b_{3_i})b_{3_i}-b_{3_{c_i}}=b_{3_i}\times(b_{3_i}\times b_{3_{c_i}})$. The magnitude of the attitude error vector, $\|e_{R_i}\|$ represents the sine of the eigen-axis rotation angle between $R_{c_i}$ and $R_i$. Therefore, $\|(e_{3}^{T}R_{c_i}^{T}R_i e_{3})R_i e_{3}-R_{c_i}e_{3}\|\leq\|e_{R_i}\|$ in $D_{1}$. It follows that
\begin{align}
\|(e_{3}^{T}R_{c_i}^{T}R_i e_{3})R_i e_{3}-R_{c_i}e_{3}\|&\leq\|e_{R_i}\|=\sqrt{\Psi_i(2-\Psi_i)}\nonumber\\
&\leq\{\sqrt{\psi_{1_i}(2-\psi_{1_i})}\triangleq\alpha_i\}<1,
\end{align}
therefore
\begin{align}
\|X_i\|&\leq \|A_i\|\|e_{R_i}\|\nonumber\\
&\leq\|A_i\|\alpha_i.
\end{align}
%Then, we can simplify \refeqn{test} as 
%\begin{align}\label{eqn:ssddss}
%\dot{\mathcal{V}}_{1} \leq &-\lambda_{\min}(Q) \|z_{1}\| ^{2}+c_{3} \|z_{1}\| \sum_{i=1}^{n}({\|A_i\|\|e_{R_i}\|})\nonumber\\
%&+2z_{1}^{T}P\mathds{B}\g(\xb,\dot{\xb}).
%\end{align}
We find an upper boundary for
\begin{align}\label{eqn:AA}
A_i=-K_{\xb} \xb-K_{\dot{\xb}}\dot\xb +u_{i}^{*},
\end{align}
by defining
\begin{align}
\|u_{i}^{*}\|\leq B_{1_i},
\end{align} 
for a given positive constant $B_{1}$. defining $K_{max}\in\Re$
\begin{align*}
&K_{\max}=\max\{\|K_{\xb}\|,\|K_{\dot{\xb}}\|\},
\end{align*}
and then the upper bound of $A$ is given by
\begin{align}
\|A_i\| & \leq K_{\max}(\|\xb\|+\|\dot{\xb}\|)+B_{1}\nonumber\\
&\leq 2K_{\max}\|z_{1}\|+B_{1}.\label{eqn:normA}
\end{align}
Using the above steps we can show that
\begin{align}
\|X\|&\leq \sum_{i=1}^{n}((2K_{\max}\|z_{1}\|+B_{1})\|e_{R_{i}}\|)\nonumber\\
&\leq (2K_{\max}\|z_{1}\|+B_{1})\alpha,
\end{align}
where $\alpha=\sum_{i=1}^{n}\alpha_{i}$. Then, we can simplify \refeqn{test} as 
\begin{align}\label{eqn:eq1}
\dot{\mathcal{V}}_{1} \leq& -(\lambda_{\min}(Q)-2c_{3}K_{\max}\alpha) \|z_{1}\| ^{2}\nonumber\\
&+\sum_{i=1}^{n}c_{3}B_{1}\|z_{1}\|\|e_{R_{i}}\|+2z_{1}^{T}P\mathds{B}\g(\xb,\dot{\xb}).
\end{align}
%%%%%%%%%%%%%%%%%%%%%%%%%%%%%%%%%%%%%%%%
\subsubsection{Lyapunov Candidate for the Complete System}
Let $\mathcal{V}=\mathcal{V}_{1}+\mathcal{V}_{2}$ be the Lyapunov function for the complete system. The time derivative of $\mathcal{V}$ is given by
\begin{align}
\dot{\mathcal{V}}=\dot{\mathcal{V}}_{1}+\dot{\mathcal{V}}_{2}.
\end{align}
Substituting \refeqn{eq1} and \refeqn{eq2} into the above equation
\begin{align}
\dot{\mathcal{V}}\leq& -(\lambda_{\min}(Q)-2c_{3}K_{\max}\alpha) \|z_{1}\| ^{2}+2z_{1}^{T}P\mathds{B}\g(\xb,\dot{\xb})\nonumber\\
&+\sum_{i=1}^{n}c_{3}B_{1} \|z_{1}\|\|e_{R_{i}}\|-\sum_{i=1}^{n}\lambda_{m}(W_{2_i})\|z_{2_i}\|^{2},
\end{align}
and using $\|e_{R_{i}}\|\leq \|z_{2_i}\|$, it can be written as
\begin{align}\label{eqn:finalsimp}
\dot{\mathcal{V}}\leq& -(\lambda_{\min}(Q)-2c_{3}K_{\max}\alpha) \|z_{1}\| ^{2}+2z_{1}^{T}P\mathds{B}\g(\xb,\dot{\xb})\nonumber\\
&+\sum_{i=1}^{n}c_{3}B_{1} \|z_{1}\|\|z_{2_i}\|-\sum_{i=1}^{n}\lambda_{m}(W_{2_i})\|z_{2_i}\|^{2}.
\end{align}
The $2z_{1}^{T}P\mathds{B}\g(\xb,\dot{\xb})$ term in the above equation is indefinite. The function $\g(\xb,\dot{\xb})$ satisfies
\begin{align*}
\frac{\|\g(\xb,\dot{\xb})\|}{\|z_{1}\|}\rightarrow 0\quad \mbox{as}\quad \|z_{1}\|\rightarrow 0.
\end{align*}
Then, for any $\gamma>0$ there exists $r>0$ such that
\begin{align*}
\|\g(\xb,\dot{\xb})\|<\gamma\|z_{1}\|\quad \forall\|z_{1}\|<r.
\end{align*}
Therfore
\begin{align}
2z_{1}^{T}P\mathds{B}\g(\xb,\dot{\xb})\leq 2\gamma\|P\|_{2}\|z_{1}\|^{2}.
\end{align}
Substituting the above inequality into \refeqn{finalsimp}
\begin{align}
\dot{\mathcal{V}}\leq& -(\lambda_{\min}(Q)-2c_{3}K_{\max}\alpha) \|z_{1}\| ^{2}+2\gamma\|P\|_{2}\|z_{1}\|^{2}\nonumber\\
&+\sum_{i=1}^{n}c_{3}B_{1} \|z_{1}\|\|z_{2_i}\|-\sum_{i=1}^{n}\lambda_{m}(W_{2_i})\|z_{2_i}\|^{2},
\end{align}
and rearranging
\begin{align}
\dot{\mathcal{V}}\leq -\sum_{i=1}^{n}(&\frac{\lambda_{\min}(Q)-2c_{3}K_{\max}\alpha}{n} \|z_{1}\| ^{2}\nonumber\\
&-c_{3}B_{1} \|z_{1}\|\|z_{2_i}\|+\lambda_{m}(W_{2_i})\|z_{2_i}\|^{2})\nonumber\\
&+2\gamma\|P\|_{2}\|z_{1}\|^{2},
\end{align}
we obtain
\begin{align}
\dot{\mathcal{V}}\leq-\sum_{i=1}^{n}(\zb_{i}^{T}W_{i}\zb_{i})+2\gamma\|P\|_{2}\|z_{1}\|^{2},
\end{align}
where $\zb_{i}=[\|z_{1}\|,\|z_{2_{i}}\|]^{T}\in\Re^{2}$ and
\begin{align}
W_i=\begin{bmatrix}
\frac{\lambda_{\min}(Q)-2c_{3}K_{\max}\alpha}{n}&-\frac{c_{3}B_{1_i}}{2}\\
-\frac{c_{3}B_{1_i}}{2}&\lambda_{m}(W_{2_i})
\end{bmatrix}.
\end{align}
By using $\|z_{1}\|\leq\|\zb_i\|$, we obtain
\begin{align}
\dot{\mathcal{V}}\leq -\sum_{i=1}^{n}(\lambda_{\min}(W_i)-\frac{2\gamma\|P\|_{2}}{n})\|\zb_{i}\|^{2}.
\end{align}
Choosing $\gamma<n(\lambda_{\min}(W_i))/2\|P\|_{2}$, and
\begin{align}
\lambda_{m}(W_{2_i})>\frac{n\|\frac{c_{3}B_{1_i}}{2}\|^2}{\lambda_{\min}(Q)-2c_{3}K_{\max}\alpha},
\end{align}
ensures that $\dot{\mathcal{V}}$ is negative definite. Then, the zero equilibrium is exponentially stable.

\bibliography{BibMaster}
\bibliographystyle{IEEEtran}

\end{document}